\documentclass[oneside,article]{memoir}
\usepackage{amsmath}         
\usepackage{amsthm,thmtools,thm-restate} 
\usepackage{enumitem}        
\usepackage{mathtools}

\usepackage{tikz}
  \usetikzlibrary{matrix,arrows,positioning,decorations.markings}
  
\usepackage{tikz-cd}

\usepackage{float} 
\usepackage{caption, subcaption} 

\usepackage[nobysame,alphabetic,initials]{amsrefs} 
\usepackage[hidelinks]{hyperref}                   
\usepackage[capitalize,nosort]{cleveref}           

\usepackage{amssymb}
\usepackage[mathscr]{euscript}
\usepackage{relsize}
\usepackage{stmaryrd}
\usepackage{stackengine}
\usepackage{mathdots} 

\usepackage{indentfirst} 
\usepackage{multicol}

\usepackage{lmodern}
\usepackage{libertine}
\usepackage[libertine]{newtxmath}

\usepackage{booktabs}

\usepackage[inline,nomargin]{fixme} 
  \fxsetup{targetlayout=color}

\usepackage{float} 
\usepackage{wrapfig} 
\usepackage{stackrel} 

\isopage[12]

\setlrmargins{*}{*}{1}
\checkandfixthelayout

\counterwithout{section}{chapter}
\usepackage{appendix}

\newcommand{\Cat}{\mathsf{Cat}}
\newcommand{\cSet}{\mathsf{cSet}}
\newcommand{\Graph}{\mathsf{Graph}}
\newcommand{\Set}{\mathsf{Set}}
\newcommand{\sCat}{\Cat_{\Delta}} 

\newcommand{\Kan}{\mathsf{Kan}}
\newcommand{\cKan}{\Kan_{\boxcat}}
\newcommand{\sKan}{\Kan_{\Delta}}

\newcommand{\sreali}[1]{\lvert #1 \rvert_{\Simp}} 
\newcommand{\creali}[1]{\lvert #1 \rvert_{\boxcat}} 
\newcommand{\eid}[1][]{\underline{\mathrm{id}}_{#1}} 

\newcommand{\cCat}{\Cat_{\boxcat}} 
\newcommand{\cohnerve}[1][\boxcat]{{\nerve}_{#1}} 

\newcommand{\cSusp}{\Sigma_{\boxcat}} 
\newcommand{\cubmc}[1]{\mathsf{#1}} 

\newcommand{\face}[2]{\partial^{#1}_{#2}} 
\newcommand{\degen}[2]{\sigma^{#1}_{#2}} 

  \newcommand{\boxcat}{\mathord{\square}} 
  
  \newcommand{\conn}[2]{\gamma^{#1}_{#2}} 
  \newcommand{\cube}[1]{\mathord{\square^{#1}}} 
  \newcommand{\dfobox}[1][n]{\mathord{\sqcap^{#1}_{i,\varepsilon}}} 

\newcommand{\restr}[2]{{#1}|_{#2}} 

  \DeclareFontFamily{U}{dmjhira}{}
  \DeclareFontShape{U}{dmjhira}{m}{n}{ <-> dmjhira }{}

  \DeclareRobustCommand{\yo}{\text{\usefont{U}{dmjhira}{m}{n}\symbol{"48}}}

  \newcommand{\adj}{\dashv}

  \newcommand{\iso}{\mathrel{\cong}}

  \renewcommand{\equiv}{\simeq}

  \newcommand{\Ho}{\operatorname{Ho}}

  \newcommand{\op}{{\mathord\mathrm{op}}}

  \newcommand{\ob}{\operatorname{ob}}

  \newcommand{\colim}{\operatorname*{colim}}
  \newcommand{\hocolim}{\operatorname*{hocolim}}

  \newcommand{\slice}{\mathbin\downarrow}

  \newcommand{\Ex}{\operatorname{Ex}}

  \newcommand{\nerve}{\operatorname{N}}

  \newcommand{\Sk}{\operatorname{Sk}}

  \newcommand{\bd}{\partial}

  \newcommand{\gprod}{\otimes}

  \newcommand{\str}{\mathfrak{C}}

  \newcommand{\id}[1][]{\operatorname{id}_{#1}}
  
  \newcommand{\simp}[1]{\mathord\Delta^{#1}}

  \newcommand{\uvar}{\mathord{\relbar}}

  \renewcommand{\tilde}{\widetilde}

  \newcommand{\cat}[1]{\mathscr{#1}}

  \newcommand{\ccat}[1]{\mathsf{#1}}

  \newcommand{\ncat}[1]{\mathsf{#1}}

  \newcommand{\Simp}{\Delta}

  \newcommand{\sSet}{\ncat{sSet}}

  \newcommand{\from}{\colon}

  \newcommand{\ito}{\hookrightarrow}

  \declaretheorem[style=definition,within=section]{definition}
  \declaretheorem[style=definition,numberlike=definition]{example}
  \declaretheorem[style=definition,numberlike=definition]{remark}

  \declaretheorem[style=plain,numberlike=definition]{corollary}
  
  \declaretheorem[style=plain,numberlike=definition]{proposition}
  \declaretheorem[style=plain,numberlike=definition]{theorem}

  \declaretheorem[style=plain,numbered=no,name=Theorem]{theorem*}
  \declaretheorem[style=plain,numbered=no,name=Conjecture]{conjecture*}

  \Crefname{corollary}{Corollary}{Corollaries}
  \Crefname{definition}{Definition}{Definitions}
  \Crefname{lemma}{Lemma}{Lemmas}
  \Crefname{proposition}{Proposition}{Propositions}
  \Crefname{remark}{Remark}{Remarks}
  \Crefname{theorem}{Theorem}{Theorems}
  \Crefname{conjecture}{Conjecture}{Conjectures}
  \Crefname{notation}{Notation}{Notations}
  \Crefname{convention}{Convention}{Conventions}

  \numberwithin{equation}{section}
  \numberwithin{figure}{section}
  \theoremstyle{plain}
  
  \theoremstyle{definition}
  
  \theoremstyle{remark}
  
  \theoremstyle{plain}
  
  \theoremstyle{plain}
  
  \theoremstyle{plain}

  \providecommand{\corollaryname}{Corollary}
  \providecommand{\definitionname}{Definition}
  \providecommand{\lemmaname}{Lemma}
  \providecommand{\propositionname}{Proposition}
  \providecommand{\remarkname}{Remark}
  \providecommand{\theoremname}{Theorem}

  \newlist{axioms}{enumerate}{1}
  \Crefname{axiomsi}{}{}

  \newenvironment{tikzeq*}
  {
    \begingroup
    \begin{equation*}
    \begin{tikzpicture}[baseline=(current bounding box.center)]
  }
  {
    \end{tikzpicture}
    \end{equation*}
    \endgroup
    \ignorespacesafterend
  }

  \tikzset
  {
    diagram/.style=
    {
      matrix of math nodes,
      column sep={4.3em,between origins},
      row sep={4em,between origins},
      text height=1.5ex,
      text depth=.25ex
    },
    over/.style={preaction={draw=white,-,line width=6pt}},
    every to/.style={font=\footnotesize},
    inj/.style={right hook->},
    surj/.style={-{Latex[open]}},
    cof/.style={>->},
    fib/.style={->>},
  }

  \DeclareFontFamily{U}{mathx}{\hyphenchar\font45}

  \DeclareFontShape{U}{mathx}{m}{n}{
    <5> <6> <7> <8> <9> <10>
    <10.95> <12> <14.4> <17.28> <20.74> <24.88>
    mathx10}{}

  \DeclareSymbolFont{mathx}{U}{mathx}{m}{n}

  \DeclareFontFamily{U}{mathb}{\hyphenchar\font45}

  \DeclareFontShape{U}{mathb}{m}{n}{
    <5> <6> <7> <8> <9> <10>
    <10.95> <12> <14.4> <17.28> <20.74> <24.88>
    mathb10}{}

  \DeclareSymbolFont{mathb}{U}{mathb}{m}{n}

  \DeclareMathSymbol{\Rsh}{\mathrel}{mathb}{"E9}

  \DeclareFontFamily{U}{MnSymbolA}{}

  \DeclareFontShape{U}{MnSymbolA}{m}{n}{
    <-6> MnSymbolA5
    <6-7> MnSymbolA6
    <7-8> MnSymbolA7
    <8-9> MnSymbolA8
    <9-10> MnSymbolA9
    <10-12> MnSymbolA10
    <12-> MnSymbolA12}{}

  \DeclareSymbolFont{MnSyA}{U}{MnSymbolA}{m}{n}

  \DeclareMathSymbol{\twoheaddownarrow}{\mathrel}{MnSyA}{27}

  \newcommand{\MSC}[1]{%
    \let\thempfn\relax
    \footnotetext[0]{2020 Mathematics Subject Classification: #1.}
  }
\usepackage{stackrel}
\usepackage{mleftright}
\mleftright
\DeclareMathSymbol{:}{\mathpunct}{operators}{"3A}
\newcommand{\notehelper}[3]{\textcolor{#3}{$\blacksquare$}\marginpar{\ifodd\thepage\raggedright\else\raggedleft\fi\color{#3}\tiny \textbf{#2:} #1}}

\tikzcdset{scale cd/.style={every label/.append style={scale=#1},
    cells={nodes={scale=#1}}}}
\usetikzlibrary{calc}
\usetikzlibrary{decorations.pathmorphing}
\usetikzlibrary{spath3}

\tikzset{curve/.style={settings={#1},to path={(\tikztostart)
    .. controls ($(\tikztostart)!\pv{pos}!(\tikztotarget)!\pv{height}!270:(\tikztotarget)$)
    and ($(\tikztostart)!1-\pv{pos}!(\tikztotarget)!\pv{height}!270:(\tikztotarget)$)
    .. (\tikztotarget)\tikztonodes}},
    settings/.code={\tikzset{quiver/.cd,#1}
        \def\pv##1{\pgfkeysvalueof{/tikz/quiver/##1}}},
    quiver/.cd,pos/.initial=0.35,height/.initial=0}

\tikzset{between/.style n args={2}{/tikz/spath/at end path construction={
    \tikzset{spath/split at keep middle={current}{#1}{#2}}
}}}

\tikzset{tail reversed/.code={\pgfsetarrowsstart{tikzcd to}}}
\tikzset{2tail/.code={\pgfsetarrowsstart{Implies[reversed]}}}
\tikzset{2tail reversed/.code={\pgfsetarrowsstart{Implies}}}
\tikzset{no body/.style={/tikz/dash pattern=on 0 off 1mm}}

\global\long\def\pr#1{\left(#1\right)}%
\global\long\def\CS{\mathsf{cSet}}%
\global\long\def\SS{\mathsf{sSet}}%
\global\long\def\Fun{\operatorname{Fun}}%
\global\long\def\hocolim{\operatorname{hocolim}}%
\global\long\def\pr#1{\left(#1\right)}%
\global\long\def\id{\operatorname{id}}%
\global\long\def\comma{\downarrow}%
\global\long\def\lim{\operatorname{lim}}%
\global\long\def\colim{\operatorname{colim}}%
\global\long\def\hocolim{\operatorname{hocolim}}%
\renewcommand{\colim}{\operatorname*{colim}}
\renewcommand{\hocolim}{\operatorname*{hocolim}}
\renewcommand{\id}[1][]{\mathrm{id}_{#1}}

\tikzstyle{vertex}=[circle, fill, minimum size=5pt, inner sep=0pt]
\tikzstyle{openvertex}=[circle, draw, fill=white, line width=0.85pt, opacity=0.6, minimum size=5pt, inner sep=0pt]
\tikzstyle{colorvertex}=[circle, draw, minimum size=6pt, inner sep=0pt]
\definecolor{vertex1}{RGB}{255,0,0}
\definecolor{vertex2}{RGB}{0,0,255}
\definecolor{vertex3}{RGB}{0,255,0}
\definecolor{vertex4}{RGB}{255,255,0}
\definecolor{vertex5}{RGB}{255,0,255}

\newcommand{\eps}{\varepsilon}

\renewcommand{\cat}[1]{\mathcal{#1}}

\newcommand{\scat}[1]{\mathsf{#1}} 
\renewcommand{\ccat}[1]{\mathsf{#1}} 
\newcommand{\ecat}[1]{\mathsf{#1}} 

\newcommand{\adjunct}[4]{#1 \from #3 \rightleftarrows #4 : \! #2} 

\newcommand{\catname}[1]{\mathsf{#1}} 

\newcommand{\PSh}{\catname{PSh}}

\renewcommand{\SS}{\mathcal{S}}

\newcommand{\diag}{\operatorname{diag}}

\thanksmarkseries{arabic}

\begin{document}

\author{Kensuke Arakawa \and Daniel Carranza \and Krzysztof Kapulkin}
\title{Derived mapping spaces of $\infty$-categories}

\maketitle

 \begin{abstract} 
   We prove the Derived Mapping Space Lemma, which generalizes the central theorem of Cisinski's work on calculus of fractions for $\infty$-categories, and allows us to provide a unified framework for analyzing mapping spaces in localizations of ($\infty$-)categories.
   As an application, we give a sufficient condition for when a cubical or simplicial category is the localization of its underlying category at homotopy equivalences.
 \end{abstract}
 
 \section*{Introduction}

Localizations of ($\infty$-)categories are of key interest in homotopy theory.
Given an ($\infty$-)category $\cat{C}$ along with a class of maps $W$, thought of as weak equivalences, we can construct an $\infty$-category $\cat{C}[W^{-1}]$ that universally inverts all maps from $W$.
The standard $1$-categorical localization, as considered in \cite{gabriel-zisman,quillen:book}, arises then as the homotopy category of $\cat{C}[W^{-1}]$, but the $\infty$-category $\cat{C}[W^{-1}]$ itself carries additional information about the mapping spaces between any pair of objects.

Over the years, many techniques have been developed to access these mapping spaces.
Historically, the first such technique was to guess-and-check, i.e., work with simplicial model categories, present in Quillen's seminal text \cite{quillen:book}, where one starts with a simplicial enrichment and verifies that it recovers the correct mapping spaces using the SM7 axiom.
However, the homotopy types of these mapping spaces can be recovered without the explicit simplicial enrichment using frames, as shown in \cite{hovey,Hirschhorn}.

Another incarnation of this problem comes from localizations of more general simplicial and cubical categories that are not necessarily part of a simplicial model category structure.
The question then becomes: when do they present the localization of their underlying $1$-category at the class of homotopy equivalences?
Here, by `homotopy equivalences,' we mean maps with an inverse up to a zig-zag of $1$-simplices/cubes.
For example, the work of Hinich \cite{hinich:dwyer-kan-localization-revisited} implies that this is the case for categories of fibrant-cofibrant objects in a model category.
In a different direction, Lurie gives a general way of recognizing this phenomenon \cite[Prop.~1.3.4.7]{lurie:higher-algebra}.
Both of these results are generally confined to the setting of simplicial categories, and it is at the very least unclear how their proofs in the cubical case (and the non-symmetric monoidal product) could be adapted, which is unfortunate because of a potential application to discrete homotopy theory \cite{carranza-kapulkin-kim}.

Lastly, the question of computing localizations of $\infty$-categories was recently considered by Cisinski \cite[Ch.~7]{cisinski:higher-categories} and the last two authors along with Lindsey \cite{carranza-kapulkin-lindsey}.
They both generalize the work of Gabriel--Zisman to $\infty$-categories, albeit in different ways: \cite[Ch.~7]{cisinski:higher-categories} describes a general framework for working with localizations, while \cite{carranza-kapulkin-lindsey} describes a specific and computationally efficient model of the localization.

The goal of this paper is to give a catch-all statement that subsumes this disparate work via one statement that instantiates to all of the above.
Specifically, our main result is the following:

\begin{theorem*}[Derived Mapping Space Lemma, cf.~\cref{thm:DMSL}]
  Let $(\cat{C}, W)$ be an ($\infty$-)category with a class of maps.
  If $Y_\bullet \from \cat{I} \to Y \slice \cat{C}$ is a resolution of an object $Y \in \cat{C}$ then there is an equivalence
  \[ \cat{C}[W^{-1}](X, Y) \equiv \hocolim\limits_{i \in \cat{I}} \cat{C}(X, Y_i)\text{,} \]
  natural in $X$.
\end{theorem*}

Here, by a resolution, we mean a homotopically constant diagram defined on a weakly contractible category that is suitably compatible with $Y$ under taking colimits.
The Derived Mapping Space Lemma was in fact inspired by, and generalizes, \cite[Thm.~7.2.8]{cisinski:higher-categories}. 
Although our generalization can be reduced to the statement of \cite[Thm.~7.2.8]{cisinski:higher-categories}, as explained in \cref{rmk:relation-to-cisinski}, the proof we give is very different from Cisinski's original proof.

The Derived Mapping Space Lemma instantiates to a variety of settings; upon proving it, we discuss in turn:
\begin{itemize}
  \item how to recover that model-categorical frames compute the correct mapping space (\cref{cor:frames});
  \item how to recognize when a cubical category is the localization of its underlying category at the class of homotopy equivalences (\cref{cubical-cat-is-localization}), and how to deduce the analogous result for simplicial categories (\cref{simplicial-cat-is-localization}) and $2$-categories (\cref{localization-2-categories});
  \item how to prove that the cubical category of graphs, introduced in \cite{carranza-kapulkin:cubical-graphs}, is indeed the localization of the category of graphs at the class of homotopy equivalences (\cref{localization-graphs}).
\end{itemize}
The last point on this list closes a considerable gap in the resolution \cite{carranza-kapulkin-kim} of an open problem from the workshop on ``Discrete and Combinatorial Homotopy Theory'' at the American Institute of Mathematics in 2023.
This also follows from \cref{cubical-cat-is-localization}, further showcasing the utility of cubical categories as a model for $\infty$-categories.
Case in point: the result for simplicial categories (\cref{simplicial-cat-is-localization}) is a simple consequence of the one for cubical categories (which is not true in the reverse direction), but the latter covers additional cases, like the cubical category of graphs.

All in all, the Derived Mapping Space Lemma is both easy to state and prove, while providing powerful consequences that were otherwise unknown or required quite advanced and specialized techniques to establish.

\paragraph{Organization and style.}
The paper is organized in a fairly straightforward manner: in \cref{sec:preliminaries}, we introduce the necessary vocabulary and establish prerequisites to state and prove our main theorem; in \cref{sec:dmsl}, we state and prove the Derived Mapping Space Lemma; and in \cref{sec:consequences}, we discuss all of the aforementioned corollaries.
Given the heavy use of cubical sets and cubical categories, we have collected the necessary background on these topics in \cref{sec:cubical-sets,sec:cubical-categories}, respectively.
In particular, the key results of \cref{sec:cubical-categories} establish the existence of the cubical analogue of the Bergner model structure, its equivalence with the Joyal model structure, and a ``helper lemma'' showing that mapping spaces of a cubical category agree with the mapping spaces of its associated quasicategory.

Throughout the paper, we work with $\infty$-categories in a model-independent way, although all of our results can be substantiated by the model in quasicategories \cite{joyal:theory-of-qcats,lurie:htt,cisinski:higher-categories}.

 \section{Preliminaries} \label{sec:preliminaries}

In this section, we collect the necessary background to state and prove the Derived Mapping Space Lemma (\cref{thm:DMSL}).

\paragraph{Size issues.} 
Whenever we speak of a general $\infty$-category $\cat{C}$, we assume the existence of a large enough Gro\-then\-dieck universe $\mathcal{U}$ so that $\cat{C}$ is small with respect to $\cat{U}$.
The category of $\infty$-groupoids should always intepreted as the category of $\mathcal{U}$-small $\infty$-groupoids so that the category of ($\infty$-)presheaves $\PSh(\cat{C})$ is locally $\mathcal{U}$-small.
The terms \emph{complete} and \emph{cocomplete} are with respect to $\mathcal{U}$-small colimits.
By the Gabriel--Ulmer duality, all statements remain true if $\cat{C}$ is not $\mathcal{U}$-small, but nonetheless $\mathcal{U}$-presentable.

Recall that a \emph{relative $\infty$-category} is a pair $(\cat{C}, W)$ where $\cat{C}$ is an $\infty$-category and $W$ is a subcategory of $\cat{C}$.
We refer to arrows of $W$ as \emph{weak equivalences}.
\begin{definition}
  Given a relative $\infty$-category $(\cat{C}, W)$, a \emph{localization} of $\cat{C}$ at $W$ is a functor
  \[ \gamma \from \cat{C} \to \cat{C}[W^{-1}] \]
  with the following properties:
  \begin{enumerate}
    \item $\gamma$ sends weak equivalences to equivalences of $\cat{C}[W^{-1}]$; and
    \item for any $\infty$-category $\cat{D}$, the pre-composition functor
    \[ \gamma^* \from \Fun(\cat{C}[W^{-1}], \cat{D}) \to \Fun(\cat{C}, \cat{D}) \]
    is full and faithful, with essential image being those functor $\cat{C} \to \cat{D}$ which send weak equivalences to equivalences.
  \end{enumerate}
\end{definition} 
Localizations of relative $\infty$-categories always exist, and are unique up to a contractible space of choices \cite[tags 01N0, 01N1]{kerodon}.
Given an object $X$ or a morphism $f$ in $\cat{C}$, we follow the convention of writing $\gamma(X)$ and $\gamma(f)$ as simply $X$ and $f$; this is justified by the fact that there always exists a localization which is a bijection on objects and injective on morphisms.

We write $\SS$ for the $\infty$-category of $\infty$-groupoids.
For an $\infty$-category $\cat{C}$, we write $\PSh(\cat{C})$ for the category of \emph{$\infty$-presheaves} on $\cat{C}$, i.e.\ functors from $\cat{C}^\op$ to $\infty$-groupoids.

The $\infty$-category $\PSh(\cat{C})$ is complete and cocomplete.
Given any functor $\varphi \from \cat{C} \to \cat{D}$, the pre-composition functor
\[ \varphi^* \from \PSh(\cat{D}) \to \PSh(\cat{C}) \]
admits both adjoints, given by left and right Kan extensions.
We denote the left adjoint by $\varphi_! \from \PSh(\cat{C}) \to \PSh(\cat{D})$.
We will not make use of the right adjoint.

The left adjoint $\varphi_! \from \PSh(\cat{C}) \to \PSh(\cat{D})$ admits an explicit description on representables:
\begin{proposition}[{\cite[Proposition 5.2.6.3]{lurie:htt}}] \label{left-adjoint-representables}
  The square
  \[ \begin{tikzcd}
    \cat{C} \ar[r, "\yo"] \ar[d, "\varphi"'] & \PSh(\cat{C}) \ar[d, "\varphi_!"] \\
    \cat{D} \ar[r, "\yo"] & \PSh(\cat{D})
  \end{tikzcd} \]
  commutes up to natural equivalence.
\end{proposition}

In the case that $\varphi$ is a localization, the pre-composition functor is full and faithful.
\begin{proposition} \label{gamma-full-faithful}
  Let $\gamma \from \cat{C} \to \cat{C}[W^{-1}]$ be a localization.
  The functor $\gamma^* \from \PSh(\cat{C}[W]^{-1}) \to \PSh(\cat{C})$ is full and faithful, and its essential image is given by functors $\cat{C}^\op \to \SS$ which send weak equivalences to equivalences of $\infty$-groupoids.
\end{proposition}
\begin{proof}
  The functor $\gamma^\op \from \cat{C}^\op \to \cat{C}[W^{-1}]^\op$ is a localization of $(\cat{C}^\op, W^\op)$, from which this follows by the definition of a localization.
\end{proof}

Lastly, recall that an $\infty$-category $\cat{I}$ is \emph{weakly contractible} if the localization of the pair $(\cat{I}, \cat{I})$ is a contractible groupoid.
\begin{proposition}[{\cite[Corollary 4.4.4.10]{lurie:htt}}]
  If $\cat{I}$ is weakly contractible and $\mathsf{const}_X \from \cat{I} \to \cat{C}$ is the functor which is constant at an object $X \in \cat{C}$ then the cone diagram  which is constant at $X$ is a homotopy colimit cone for $\mathsf{const}_X$. \qed
\end{proposition}

\section{The Derived Mapping Space Lemma} \label{sec:dmsl}

We recall the following phenomenon from the theory of model categories: for a model category $\cat{M}$, one can always form the category of simplicial objects $\Fun(\Delta^\op, \cat{M})$, equipped with the Reedy model structure.
Given an object $Y \in \cat{M}$, taking a fibrant replacement of the constant simplicial object at $Y$ produces a new simplicial object $Y_\bullet \from \Delta^\op \to \cat{M}$, 
known as a \emph{simplicial resolution} of $Y$.
Each object in the resolution $Y_n$ admits a weak equivalence from $Y$, and these maps $ Y \to Y_n$ form a cone over $Y_\bullet$. 
Moreover, for any $X \in \cat{M}$ cofibrant, the simplicial set $\cat{M}(X, Y_\bullet)$ has the homotopy type of the derived mapping space, i.e.\ it computes the mapping space of the localization $\cat{M}[W^{-1}]$.

The \emph{derived mapping space lemma} generalizes this setup, giving sufficient conditions for when the mapping space into a ``resolution'' (before localizing) correctly computes the mapping space after localizing.

Before stating the lemma, we axiomatize the properties we need of a resolution.
\begin{definition} \label{def:I-resolution}
  Let $(\cat{C}, W)$ be a relative $\infty$-category and let $Y$ be an object of $\cat{C}$.
  For a small $\infty$-category $\cat{I}$, an \emph{$\cat{I}$-resolution} of $Y$ is a functor $Y_\bullet \from \cat{I} \to Y \slice \cat{C}$ into the slice under $Y$ such that:
  \begin{enumerate}
    \item[(R1)] $\cat{I}$ is weakly contractible;
    \item[(R2)] for all $i \in \cat{I}$, the structure morphism $Y \to Y_i$ maps to an equivalence in $\cat{C}[W^{-1}]$; 
    \item[(R3)] the colimit presheaf
    \[ \colim\limits_{\cat{I}} \cat{C}(-, Y_{\bullet}) = \colim \big( \cat{I} \xrightarrow{Y_\bullet} Y \slice \cat{C} \xrightarrow{\Pi} \cat{C} \xrightarrow{\yo} \PSh(\cat{C}) \big) \]
    sends weak equivalences to equivalences of $\infty$-groupoids.
  \end{enumerate}
\end{definition}
Let $Y_\bullet$ be an $\cat{I}$-resolution of an object $Y \in \cat{C}$.
For each object $i \in \cat{I}$, let $\phi_i$ denote the composite 
\[
	\cat{C}(-,Y) \to \cat{C}(-,Y_i) \to \colim\limits_{\cat{I}}\cat{C}(-,Y_{\bullet}).
\]
The naturality of the unit $\eta$ of the adjunction $\gamma_!\dashv \gamma^*$ gives us a commutative diagram 

\[ \label{eq:phi-square} \begin{tikzcd}[column sep = 3em, row sep = 2em]
	\cat{C}(-, Y) \ar[r, "\phi_i"] \ar[d, "\eta"'] & \colim\limits_{\cat{I}} \cat{C}(-, Y_{\bullet}) \ar[d, "\eta", "\sim"'] \\
	\gamma^* \gamma_! \cat{C}(-, Y) \ar[r, "{\gamma^* \gamma_! \phi_i}"''] & \gamma^* \gamma_! \colim\limits_{\cat{I}} \cat{C}(-, Y_{\bullet}).
\end{tikzcd} \tag{$\ast$} \]
Notice that the right hand $\eta$ is an equivalence, as the colimit presheaf $\colim\limits_{\cat{I}} \cat{C}(-, Y_{\bullet})$ is in the essential image of the functor $\gamma^*$ by condition (R3).
By \cref{left-adjoint-representables}, we identify $\gamma_! \cat{C}(-, Y)$ as $\cat{C}[W^{-1}](-, Y)$.
With these identification, we shall write $\Phi_i$ for $\gamma^* \gamma_!{\phi_i}$, which we view as a morphism
\[ \Phi_i \from \gamma^* \cat{C}[W^{-1}](-, Y) \to \colim\limits_{\cat{I}} \cat{C}(-, Y_{\bullet}) \]
in $\PSh(\cat{C})$.

\begin{theorem}[Derived mapping space lemma] \label{thm:DMSL}
  Let $(\cat{C}, W)$ be a relative $\infty$-category.
  If $Y_\bullet \from \cat{I} \to Y \slice \cat{C}$ is a resolution of an object $Y \in \cat{C}$ then
  \[ \Phi_i \from \gamma^* \cat{C}[W^{-1}](-, Y) \to \colim\limits_{\cat{I}} \cat{C}(-, Y_{\bullet}) \]
  is an equivalence of presheaves.

\end{theorem}
The derived mapping space lemma \cref{thm:DMSL} was inspired by, and generalizes, Cisinski's theorem on calculus of fractions \cite[Theorem 7.2.8]{cisinski:higher-categories}.
We explain this in detail in \cref{rmk:relation-to-cisinski}.
\begin{proof}
	The structure morphisms $Y \to Y_i$ induce a morphism $\colim\limits_{\cat{I}} \cat{C}(-, Y) \to \colim\limits_{\cat{I}} \cat{C}(-, Y_{\bullet})$ in $\PSh(\cat{C})$, where $\colim\limits_{\cat{I}} \cat{C}(-, Y)$ denotes the colimit of the constant diagram $\cat{I}\to \PSh(\cat{C})$ at $\cat{C}(-,Y)$. 
  As $\cat{I}$ is weakly contractible, the colimit cone for $\colim\limits_{\cat{I}} \cat{C}(-, Y)$ is the constant cone.
  Thus, the left map in the commutative square
  \[ \begin{tikzcd}
	  \cat{C}(-, Y) \ar[r] \ar[d, "\sim"'] \ar[rd, "\phi_i"{description}] & \cat{C}(-, Y_i) \ar[d] \\
    \colim\limits_{\cat{I}} \cat{C}(-, Y) \ar[r] & \colim\limits_{\cat{I}} \cat{C}(-, Y_{\bullet})
  \end{tikzcd} \]
  is an equivalence.
  Applying $\gamma_!$ to this square and using the fact that $\gamma_!$ preserves colimits, we obtain the following diagram:
  \[ \begin{tikzcd}
	  \gamma_! \cat{C}(-, Y) \ar[r] \ar[d, "\sim"'] \ar[rd, "\gamma_!\phi_i"{description}] & \gamma_! \cat{C}(-, Y_i) \ar[d] \\
    \gamma_! \colim\limits_{\cat{I}} \cat{C}(-, Y) \ar[r] \ar[d, "\sim"'] & \gamma_! \colim\limits_{\cat{I}} \cat{C}(-, Y_\bullet) \ar[d, "\sim"] \\
    \colim\limits_{\cat{I}} \gamma_! \cat{C}(-, Y) \ar[r, "\alpha"'] & \colim\limits_{\cat{I}} \gamma_! \cat{C}(-, Y_{\bullet}).
  \end{tikzcd} \]
  By \cref{left-adjoint-representables}, the map $\alpha$ at the bottom is naturally equivalent to the morphism
  \[ \colim\limits_{\cat{I}} \cat{C}[W^{-1}](-, Y) \to \colim\limits_{\cat{I}} \cat{C}[W^{-1}](-, Y_{\bullet}) \] 
  induced by the morphisms $Y \to Y_i$ in $\cat{C}[W^{-1}]$.
  Since these morphisms are equivalences in $\cat{C}[W^{-1}]$, so must be $\alpha$. Hence $\gamma_!\phi_i$ is also an equivalence, from which we deduce that $\Phi_i \simeq \gamma^*\gamma_!\phi_i$ is an equivalence. 
\end{proof}
  In the context of $\infty$-categories, a \emph{resolution} of an object $Y \in \cat{C}$ may also refer to any diagram $\cat{I} \to \cat{C}$ whose colimit is $Y$.
  Under this usage, it is immediate from condition (R2) of \cref{def:I-resolution} that $Y_\bullet$ is a resolution of $Y$ in $\cat{C}[W^{-1}]$.
  \Cref{thm:DMSL} establishes a stronger result: that $\cat{C}(-, Y_\bullet)$ is a resolution of the ``derived representable'' presheaf $\gamma^* \cat{C}[W^{-1}](-, Y) \in \PSh(\cat{C})$.

\begin{remark}
 The $\infty$-category $\cat{I}$ is a connected category by (R1), hence the maps $\phi_i$ and $\Phi_i$ are all homotopic to one another.
 However, (R1) ensures something stronger: they are homotopic via \emph{essentially unique} homotopies.
 To see this, note that the maps $\phi_i$ form a natural transformation between two constant diagrams, which is exactly the data of a functor $\cat{I}\to \PSh(\cat{C})\big( \cat{C}(-,Y) , \colim\limits_{\cat{I}} \cat{C}(-,Y_\bullet) \big)$ into the mapping $\infty$-groupoid.
 As all morphisms in the target are invertible, this functor factors through the localization $\cat{I}[\cat{I}^{-1}]$.
 Condition (R1) says that this localization is the terminal $\infty$-groupoid, so all the maps $\phi_i$ are homotopic by canonical homotopies.
\end{remark}

\begin{remark} \label{rmk:relation-to-cisinski}
 \cref{thm:DMSL} is a generalization of \cite[Theorem 7.2.8]{cisinski:higher-categories} on calculus of fractions.
 More precisely, Cisinski calls a diagram $F\from\cat{A}\to \cat{C}$ a \emph{left calculus of fractions} at $Y\in \cat{C}$ if $\cat{A}$ has an initial object $\emptyset\in \cat{A}$ mapping to $Y$, and if the induced functor $\widetilde{F}\from \cat{A}\simeq \emptyset \downarrow \cat{A}\to Y\comma \cat{C}$ is a resolution in the sense of \cref{def:I-resolution}.
 His theorem, which is proven by a different technique than ours, asserts that if $F$ is a right calculus of fractions, then one has an equivalence of presheaves $\gamma^*\cat{C}[W^{-1}](-,Y)\simeq \colim\limits_{a\in \cat{A}}\cat{C}(-,F(a))$. 
 (In the broader context of \cite[Ch.~7.2]{cisinski:higher-categories}, which alludes to the work of Gabriel and Zisman \cite{gabriel-zisman},  $\cat{A}$ is taken to be a subcategory of $Y\downarrow \cat{C}$, thus explaining the use of the term ``fraction''.)

 In fact, it is also possible to derive \cref{thm:DMSL} directly from \cite[Thm.~7.2.8]{cisinski:higher-categories}. Indeed, any $\cat{I}$-resolution $\cat{I}\to Y\downarrow \cat{C}$ determines a left calculus of fractions $\{\emptyset\}\star\cat{I}\to \cat{C}$, since $\cat{I}\to \{\emptyset\}\star \cat{I}$ is final by (R1).
	However, this derivation is artificial from the ``resolution'' viewpoint (for instance, to accommodate simplicial resolutions to the setting of fractions, we would have to use the unusual indexing category $\cat{A} = \{\emptyset\}\star \Delta^\op$), nor does it give a conceptual explanation for \cref{thm:DMSL}. 
\end{remark}

\section{Consequences} \label{sec:consequences}

For our first consequence of the derived mapping space lemma, we verify that it generalizes the case of simplicial frames in a model category.
Given a model category $\cat{M}$, we write $\cat{M}^{\mathrm{c}}$ for the full subcategory consisting of cofibrant objects.
\begin{corollary} \label{cor:frames}
  Let $\cat{M}$ be a model category, and let $Y_{\bullet}$ be a simplicial
  resolution of an object $Y \in \cat{M}$.
  For every cofibrant object $X \in \cat{M}^{\mathrm{c}}$, there is an equivalence
  \[ \cat{M}[W^{-1}](X, Y) \simeq \cat{M}(X, Y_\bullet) \]
  natural in $X \in \cat{M}^{\mathrm{c}}$.
\end{corollary}
\begin{proof}
  Without loss of generality, we may assume $Y_\bullet$ is Reedy cofibrant (and fibrant).
Indeed, the non-cofibrant case follows from \cite[Thm.~16.5.2.(2)]{Hirschhorn}, which says that a Reedy cofibrant replacement $\tilde{Y_\bullet} \to Y_\bullet$ induces an equivalence (in fact, an acyclic fibration) of Kan complexes
  \[ \cat{M}(X, \tilde{Y_\bullet}) \xrightarrow{\sim} \cat{M}(C, Y_\bullet). \]
  We may also assume $Y$ is cofibrant, since every object is equivalent to a cofibrant object in the localization $\cat{M}[W^{-1}]$.

  By \cite[Thm.~7.5.18]{cisinski:higher-categories}, the induced functor $\cat{M}^{\mathrm{c}}[W^{-1}] \to \cat{M}[W^{-1}]$ is an equivalence, from which we deduce an equivalence
  \[ \cat{M}[W^{-1}](X, Y) \simeq \cat{M}^{\mathrm{c}}[W^{-1}](X, Y).  \]
  Lastly, recall (e.g. from \cite[Cor.~18.7.7]{Hirschhorn}) that every simplicial set $K$ is its own geometric realization, i.e.\
  \[ \colim \big( \Delta^\op \xrightarrow{K} \Set \ito \SS \big) \simeq K. \]
  Combining our reductions, the desired equivalence may be re-written as
  \[ \cat{M}^{\mathrm{c}}[W^{-1}](X, Y) \simeq \colim\limits_{[n] \in \Delta^\op} \cat{M}^{\mathrm{c}}(X, Y_n). \]
  To apply \cref{thm:DMSL}, we need only verify that $Y_\bullet \from \Delta^\op \to Y \slice \cat{M}^{\mathrm{c}}$ is a resolution of $Y$ in the sense of \cref{def:I-resolution}:
  \begin{itemize}
    \item[(R1)] $\Delta^\op$ is weakly contractible since it has an initial object $[0]$;
    \item[(R2)] the maps $Y \to Y_n$ are weak equivalence by the assumption that $Y_\bullet$ is a Reedy fibrant replacement of $Y$;
    \item[(R3)] if $X' \to X$ is a weak equivalence of cofibrant objects then the map $\cat{M}^{\mathrm{c}}(X, Y_\bullet ) \to \cat{M}^{\mathrm{c}}(X', Y_\bullet )$ is an equivalence by \cite[Cor.~16.5.5.(3)]{Hirschhorn}. \qedhere
  \end{itemize}
\end{proof}

A more significant consequence of the derived mapping space lemma is that it allows us to give an easy-to-verify sufficient condition for when a cubical (or simplicial) category is the localization of its underlying category at its \emph{homotopy equivalences}.

We acknowledge here that while many aspects of the theory of cubical categories might be known to experts, they are not well documented (or documented at all!) in the existing literature.
For this reason, we collect all the necessary statements in \cref{sec:cubical-categories}, which in turn relies on \cref{sec:cubical-sets} on cubical sets.
The key results are: the existence of the Bergner model structure on cubical categories (\cref{model-str-ccat}), its Quillen equivalence to the Joyal model structure on simplicial sets (\cref{str-cohnerve-quillen-equiv}) via the cubical homotopy coherent nerve of \cite{kapulkin-voevodsky}, and the fact that the mapping spaces of the homotopy coherent nerve of a locally Kan cubical category agree via triangulation with the mapping spaces of the cubical category (\cref{repr-ex-infty-t-commute}).

Prior to reading the remainder of the section, the reader unfamiliar with cubical sets should consult \cref{sec:cubical-sets,sec:cubical-categories} for the necessary background.
Readers familiar with that theory can simply treat the three results listed above (\cref{model-str-ccat,str-cohnerve-quillen-equiv,repr-ex-infty-t-commute}) as a black box and consult \cref{sec:cubical-categories} if necessary.

Returning to the question of defining the underlying category of a cubical category, by \emph{homotopy equivalence} in a cubical category, we mean a morphism which admits an inverse ``up to homotopy'', i.e.\ up to some path in the mapping space.
\begin{definition} \label{def:hpty-equiv}
   Let $\ccat{C}$ be a cubical category.
  A morphism $f \from X \to Y$ in $\ccat{C}$ is a \emph{homotopy equivalence} if there exists
  \begin{enumerate}
    \item a morphism $g \from Y \to X$;
    \item a zig-zag of 1-cubes in $\ccat{C}(X, X)$ from $\id[X]$ to $gf$; and
    \item a zig-zag of 1-cubes in $\ccat{C}(Y, Y)$ from $\id[Y]$ to $fg$.
  \end{enumerate}
\end{definition}
The following are straightforward consequences of \cref{def:hpty-equiv}.
\begin{itemize}
  \item If $f \from X \to Y$ and $g \from Y \to Z$ are homotopy equivalences then $gf \from X \to Z$ is as well.
  \item If $H$ is a 1-cube in any mapping space $\ccat{C}(X, Y)$ from a morphism $f$ to a morphism $g$ then we have homotopies
  \[ \begin{tikzcd}[cramped, column sep = 1.6em]
    \ccat{C}(W, X) \ar[r, "{\face{}{1, 0} \otimes \id}"] \ar[rdd, "{f_*}", swap, bend right] & \cube{1} \otimes \ccat{C}(W, X) \ar[d, "{H \otimes \id}"] & \ar[l, "{\face{}{1, 1} \otimes \id}", swap] \ar[ldd, "{g_*}", bend left] \ccat{C}(W, X) \\
    {} & \ccat{C}(X, Y) \otimes \ccat{C}(W, X) \ar[d, "\circ"] & {} \\
    {} & \ccat{C}(W, Y) & {}
  \end{tikzcd} \qquad \begin{tikzcd}[cramped, column sep = 1.8em]
    \ccat{C}(Y, Z) \ar[r, "{\id \otimes \face{}{1, 0}}"] \ar[rdd, "{f^*}", swap, bend right] & \ccat{C}(Y, Z) \otimes \cube{1} \ar[d, "{\id \otimes H}"] & \ar[l, "{\face{}{1, 1} \otimes \id}", swap] \ar[ldd, "{g^*}", bend left] \ccat{C}(Y, Z) \\
    {} & \ccat{C}(Y, Z) \otimes \ccat{C}(X, Y) \ar[d, "\circ"] & {} \\
    {} & \ccat{C}(X, Z) & {}
  \end{tikzcd} \]
  between the postcomposition maps $f_* \sim g_*$ and the precomposition maps $f^* \sim g^*$.
  It follows that if $f$ is a homotopy equivalence then $f^*$ and $f_*$ are homotopy equivalences.

  \item If $F \from \ccat{C} \to \ccat{D}$ is a DK-equivalence then a morphism $f$ is a homotopy equivalence if and only if $Ff$ is.
  \item If $\ccat{C}$ is locally Kan then $\id[X]$ and $gf$ are connected by a 1-cube in $\ccat{C}(X, X)$ in both directions (and likewise for $\id[Y]$ and $fg$).
  \item If $\ccat{C}$ is locally Kan then $f$ is a homotopy equivalence if and only if it becomes an equivalence in $\cohnerve \ccat{C}$; this follows from the previous fact.
  \item A morphism $f$ is a homotopy equivalence if and only if it becomes an isomorphism in the homotopy category.
\end{itemize}

In practice, a notion of ``homotopy'' for morphisms in a category $\cat{C}$ often comes from a cubical or simplicial enrichment.
For instance, several convenient categories of topological spaces admit a simplicial enrichment where the 1-simplices are precisely the ordinary homotopies between continuous functions.
Another example, which we address later in the paper, comes from the field of \emph{discrete homotopy theory}, where a notion of ``discrete homotopy'' between maps of graphs may be obtained via cubical enrichment.
In such settings, it is natural to suspect that the mapping spaces coming from the enrichment are ``homotopically correct'', meaning they compute the mapping spaces of $\cat{C}$, localized at these homotopy equivalences.
(This is true for the standard simplicial enrichments of spaces).

In contrast, it is not in general true that every cubical/simplicial category is the localization of its underlying category at homotopy equivalences. 
For example, consider a cubical/simplicial group $X$ with a non-trivial homotopy group $\pi_n$ in dimension $n \geq 1$ (e.g.\ $X$ is the singular complex of the circle $S^1$).
When viewing $X$ as a one-object cubical/simplicial category $BX$, its underlying category is the one-object groupoid $B(X_0)$, and every morphism is a homotopy equivalence (since every morphism is an isomorphism).
However, the localization of $B(X_0)$ is the Eilenberg-MacLane space $K(X_0, 1)$, thus the mapping spaces $BX(\ast, \ast) = X$ and $B(X_0)[W^{-1}](\ast, \ast) = X_0$ fail to be homotopy-equivalent.

We remark that if $\ccat{C}$ is a cubical/simplicial category then $\ccat{C}_0[W^{-1}]$ depends only on the 1-skeleton of the mapping spaces of $\ccat{C}$.
Thus, another class of counterexamples comes from altering the mapping spaces of $\ccat{C}$ in any way that preserves the 1-skeleta, but which changes the homotopy type.

To give a sufficient condition for when $\ccat{C}$ correctly computes the localization at homotopy equivalences, we will need the notion of \emph{weak (co)tensors}.
\begin{definition}
  Let $(\cat{V}, \otimes)$ be a right-closed monoidal category.
  and write $\underline{\cat{V}}(A, -)$ for the right adjoint to the functor $- \otimes A \from \cat{V} \to \cat{V}$.
  Suppose $\ecat{C}$ is a $\cat{V}$-category.
  Given objects $A \in \cat{V}$ and $X \in \ecat{C}$,
  \begin{enumerate}
    \item a \emph{weak tensor} of $X$ by $A$ is an object $A \otimes X \in \ccat{C}$ together with an isomorphism of $\cat{V}$-objects
    \[ \ecat{C}(A \otimes X, Y) \cong  \underline{\cat{V}} \big( A, \ecat{C}(X, Y) \big) \]
    which is unenriched natural in $Y \in \ecat{C}_0$.
    \item a \emph{weak cotensor} of $X$ by $A$ is an object $A \pitchfork X \in \ccat{C}$ together with an isomorphism of $\cat{V}$-objects
    \[ \ecat{C}(Z, A \pitchfork X) \cong \underline{\cat{V}} \big( A, \ecat{C}(Z, X) \big) \]
    which is unenriched natural in $Z \in \ecat{C}_0^\op$.
  \end{enumerate}
\end{definition}
We say that $\ecat{C}$ \emph{admits weak tensors} (analogously, $\ecat{C}$ \emph{admits weak cotensors}) by an object $A \in \cat{V}$ if, for every object of $\ecat{C}$, there exists a weak tensor (analogously, a weak cotensor) of $X$ by $A$.

The following examples highlight various ways of constructing weak (co)tensors.
\begin{example}
  Any $\cat{V}$-category admits weak (co)tensors by the monoidal unit $I$ by setting $I \otimes X = X = I \pitchfork X$.
  The naturality condition ensures that weak (co)tensors are unique up to an isomorphism in $\cat{C}_0$ if they exist, so in practice, we assume $I \otimes X = X$ and $I \pitchfork X = X$ for convenience.
\end{example}
\begin{example} \label{ex:self-enriched-tensor}
  If $(\cat{V}, \otimes)$ is right-closed then $\cat{V}$ itself becomes a $\cat{V}$-category.
  In this case, $\cat{V}$ admits all weak tensors, and they are given by left multiplication $A \otimes X$.
  If an object $A \in \cat{V}$ admits natural isomorphisms $A \otimes X \cong X \otimes A$ for all $X$ then $\cat{V}$ admits weak cotensors by $A$, and they are given by the right-closed structure $\underline{\cat{V}}(A, X)$.
\end{example}
\begin{example} \label{ex:change-of-base-tensor}
  Suppose $\adjunct{L}{R}{\cat{V}}{\cat{W}}$ is an adjunction between monoidal categories where the left adjoint $L$ is strong monoidal.
  For a $\cat{W}$-category $\ecat{D}$, let $R_\bullet \ecat{D}$ denote the ``change-of-base'' of $\ecat{D}$ to a $\cat{V}$-category.
  Given $A \in \cat{V}$, if $X$ admits a weak tensor by $LA$ in $\ecat{D}$ then $LA \otimes X$ is a weak tensor of $X$ by $A$ in $R_\bullet \ecat{D}$.
  The analogous statment holds for the weak cotensor $LA \pitchfork X$.
\end{example}
\begin{example} \label{ex:tensor-of-product}
  If $X$ admits a weak tensor by $B$ and $B \otimes X$ admits a weak tensor by $A$ then $A \otimes (B \otimes X)$ is a weak tensor of $X$ by $A \otimes B$.
  Analogously, if $X$ admits a weak cotensor by $B$ and $B \pitchfork X$ admits a weak cotensor by $A$ then $A \pitchfork (B \pitchfork X)$ is a weak cotensor of $X$ by $A \otimes B$.
\end{example}
The motivation for studying weak (co)tensors as opposed to the standard notion of (co)tensors in enriched category theory is that unenriched naturality is well-defined even if the base of enrichment $(\cat{V}, \otimes)$ is \emph{not} symmetric (nor braided).
In particular, we wish to apply our results to the case of cubical sets with the geometric product, which is neither symmetric nor braided.

Our goal is to prove the following.
\begin{theorem} \label{cubical-cat-is-localization}
  Let $\ccat{C}$ be a cubical category which admits weak tensors or weak cotensors by $\cube{1}$.
  Then, for any fibrant replacement $\ccat{C} \to \ccat{C}^\mathrm{f}$ of $\ccat{C}$ in the cubical Bergner model structure, the map
  \[ \ccat{C}_0 \ito \cohnerve \ccat{C} \to \cohnerve \ccat{C}^{\mathrm{f}} \]
  is a localization of $\ccat{C}_0$ at homotopy equivalences. 
\end{theorem}
Before proving \cref{cubical-cat-is-localization}, we develop some auxiliary results on cubical sets and cubical weak (co)tensors.

A consequence of the Yoneda lemma is that weak (co)tensors are automatically (unenriched) functorial in the variable $A$.
Given a morphism $f \from A \to B$ in $\cat{V}$, if $X$ admits weak tensors by both $A$ and $B$ then we obtain a dotted arrow filling the square
\[ \begin{tikzcd}
	\cat{C}_0(X \otimes B, Y) \ar[r, "\cong"] \ar[d, dotted] & {\underline{\cat{V}}} \big( B, \cat{C}(X, Y) \big) \ar[d, "f^*"] \\
	\cat{C}_0(X \otimes A, Y) \ar[r, "\cong"] & {\underline{\cat{V}}}\big( A, \cat{C}(X, Y) \big)
\end{tikzcd} \]
By the Yoneda lemma, the dotted arrow is given by pre-composition with a morphism $X \otimes A \to X \otimes B$ in $\cat{C}_0$, which we denote by $X \otimes f$.
An analogous argument can be used to construct a morphism $f \pitchfork X \from B \pitchfork X \to A \pitchfork X$.
The uniqueness of the dotted arrow gives compatibility with composition and identities, which proves the following:
\begin{proposition} \label{tensors-are-functorial}
  Let $\cat{A} \subseteq \cat{V}$ be any subcategory.
  \begin{enumerate}
    \item If $X$ admits tensors by all objects of $\cat{A}$, then the assignments $A \mapsto X \otimes A$ and $f \mapsto X \otimes f$ form a functor
    \[ X \otimes - \from \cat{A} \to \cat{C}_0. \]
    Moreover, the isomorphism 
    $ \cat{C}(A \otimes X, Y) \cong  \underline{\cat{V}} \big( A, \cat{C}(X, Y) \big)$
    is natural in $A \in \cat{A}^\op$ and $Y \in \cat{C}_0$.
    \item If $X$ admits cotensors by all objects of $\cat{A}$, then the assignments $A \mapsto A \pitchfork X$ and $f \mapsto f \pitchfork X$ form a functor
    \[ - \pitchfork X \from \cat{A}^\op \to \cat{C}_0. \]
    Moreover, the isomorphism
    $ \cat{C}(Z, A \pitchfork X) \cong \underline{\cat{V}} \big( A, \cat{C}(Z, X) \big) $
    is natural in $A \in \cat{A}^\op$ and $Z \in \cat{C}_0^\op$. \qed
  \end{enumerate}
\end{proposition}

We now consider the case where $(\cat{V}, \otimes)$ is the category of cubical sets with the geometric product $(\cSet, \otimes)$.
By \cref{ex:tensor-of-product}, if a cubical category $\ccat{C}$ admits weak tensors by $\cube{1}$ then it admits weak tensors by $\cube{n}$ for all $n$.
By \cref{tensors-are-functorial}, these weak tensors form a cocubical object $\cube{\bullet} \otimes X \from \boxcat \to \ccat{C}_0$.
We show the cubical structure map $\cube{n} \otimes X \to \cube{0} \otimes X = X$ is a homotopy equivalence, showcasing our first interaction between enriched structure and homotopical structure.
\begin{proposition} \label{tensor-cubical-htpy-equiv}
  Let $\ccat{C}$ be a cubical category.
  \begin{enumerate}
    \item If $\ccat{C}$ admits weak tensors by $\cube{1}$ then the map
    \[ \begin{tikzcd}
      \cube{n} \otimes X \ar[r, "{\mathord{!} \otimes X}"] & \cube{0} \otimes X = X
    \end{tikzcd} \]
    is a homotopy equivalence.
    \item If $\ccat{C}$ admits weak cotensors by $\cube{1}$ then the map
    \[ \begin{tikzcd}
      X = \cube{0} \otimes X \ar[r, "{\mathord{!} \pitchfork X}"] & \cube{n} \otimes X
    \end{tikzcd} \]
    is a homotopy equivalence.
  \end{enumerate}
\end{proposition}
\begin{proof}
  We prove (1), as the proof of (2) is analogous. 
  
  Since homotopy equivalences compose, it suffices to show the map $\degen{}{1} \otimes X \from \cube{n+1} \otimes X \to \cube{n} \otimes X$ is a homotopy equivalence for all $n \geq 0$.
  This map has a section given by $\face{}{1,1} \otimes X$, so it remains to construct a homotopy $\id \sim \face{}{1,1}\degen{}{1} \otimes X$.
  
  By \cref{tensors-are-functorial}, we have a commutative diagram
  \[ \begin{tikzcd}
    \cube{n+1} \otimes X \ar[d, "{\face{}{1,0} \otimes X}"'] \ar[rd, bend left, "\id"] & {} \\
    \cube{n+2} \otimes X \ar[r, "{\conn{}{1,0} \otimes X}"] & \cube{n+1} \otimes X \\
    \cube{n+1} \otimes X \ar[u, "{\face{}{1,1 \otimes X}}"] \ar[ur, bend right, "\face{}{1,1}\degen{}{1} \otimes X"'] & {}
  \end{tikzcd} \]
  From the isomorphism
  \begin{align*} 
    \ccat{C}(\cube{n+2} \otimes X, \cube{n+1} \otimes X) &\cong \underline{\cSet} \big( \cube{n+2}, \ccat{C}(X, \cube{n+1} \otimes X) \big) \\
    &\cong \underline{\cSet} \big( \cube{1} \otimes \cube{n+1}, \ccat{C}(X, \cube{n+1} \otimes X) \big) \\
    &\cong \underline{\cSet} \big( \cube{1}, \ccat{C}(\cube{n+1} \otimes X, \cube{n+1} \otimes X) \big),
  \end{align*}
  we see the morphism $\conn{}{1,0} \otimes X$ determines a 1-cube in the mapping space $\ccat{C}(\cube{n+1} \otimes X, \cube{n+1} \otimes X)$, whereas the arrows $\id$ and $\face{}{1,1}\degen{}{1}$ determine 0-cubes.
  Applying naturality at the maps $\face{}{1,0} \otimes X$ and $\face{}{1,1} \otimes X$ gives the desired endpoint equalities for the 1-cube $\conn{}{1,0} \otimes X$.
\end{proof}
\begin{remark}
  The proof of \cref{tensor-cubical-htpy-equiv} makes use of negative connections.
  For categories enriched in cubical sets without connections, the result still holds.
  Any cubical category is DK-equivalent to a cubical category with connections, and the notion of homotopy equivalence is invariant under DK-equivalence.
\end{remark}

The final ingredient in the proof of \cref{cubical-cat-is-localization} is an analogue of \cite[Cor.~18.7.7]{Hirschhorn} for cubical sets.
Recall that a \emph{bicubical set} is a cubical object in $\cSet$, i.e.\ a functor $\boxcat^\op \to \cSet$; this is the same data as a functor $\boxcat^\op \times \boxcat^\op \to \Set$.
The \emph{diagonal} of a bicubical set $X$ is the cubical set $\diag X$ defined by the composite 
\[ \boxcat^\op \xrightarrow{(\id, \id)} \boxcat^\op \times \boxcat^\op \xrightarrow{X} \Set . \]
\begin{proposition} \label{diagonal-computes-hocolim}
  For a bicubical set $X \from \boxcat^\op \to \cSet$, the diagonal $\diag X$ computes the (homotopy) colimit of $X$, regarded as a diagram $\boxcat^\op \to \SS$ of $\infty$-groupoids.
  That is, there is a zig-zag of weak equivalences
  \[ \diag X \simeq \colim\limits_{[1]^n \in \boxcat^\op} X_n . \]
\end{proposition}
\begin{proof}
  The functor $\diag \from \Fun({\boxcat^\op}, \cSet) \to \cSet$ admits a right adjoint $R \from \cSet \to \Fun(\boxcat^\op, \cSet)$ which sends a cubical set $K$ to the bicubical set
  \[ (RK)_{m, n} := \cSet \big( \diag (\cube{m, n}) , K \big) = \cSet \big( \cube{m} \times \cube{n}, K \big). \]
  The functor $\diag$ sends levelwise monomorphisms to monomorphisms, and levelwise weak equivalences to weak equivalences \cite[Ex.~3.11]{carranza-kapulkin-wong:diagonal}.
  Thus, it is left Quillen with respect to the injective model structure on the functor category $\Fun(\boxcat^\op, \cSet)$. 
  (This can also be proved by rewriting the diagonal as a coend $\diag X\cong \int^{[1]^n\in \square}\square^n\times X_n$, and then using the formalism of Reedy categories \cite[Proposition A.2.9.26.]{lurie:htt}).

  For fixed $m \geq 0$, the cubical set $(RK)_{m, \bullet}$ is isomorphic to the exponential object $K^{\cube{m}}$ (i.e.\ the right adjoint to the functor $- \times \cube{m}$).
  Since the Grothendieck model structure on $\cSet$ is monoidal with respect to the categorical product \cite[Thm.~5.13]{doherty:symmetry}, $R$ sends cubical Kan complexes to homotopy-constant diagrams (i.e.\ it sends all morphisms of $\boxcat^\op$ to weak equivalences).
  This shows the total right derived functor models the constant diagram functor $\SS \to \SS^{\boxcat^\op}$, hence the total left derived functor models the $(\infty, 1)$-colimit.
  Since all objects of $\Fun(\boxcat^\op, \cSet)$ are cofibrant in the injective model structure, the diagonal functor is weakly equivalent to the homotopy colimit.
\end{proof}
\begin{remark}
  The proof of \cref{diagonal-computes-hocolim} requires at least one connection (so that $\cSet$ is a monoidal model category with respect to the categorical product).
\end{remark}
We use \cref{diagonal-computes-hocolim} to deduce that every cubical set is its own geometric realization. 
\begin{corollary} \label{cubical-set-geometric-realization}
  Let $X$ be a cubical set.
  Then, $X$ is the (homotopy) colimit of the diagram $\boxcat^\op \xrightarrow{X} \Set \ito \SS$.
\end{corollary}
\begin{proof}
  We view $[1]^n \mapsto X_n$ as a bicubical set $\boxcat^\op \to \cSet$ taking values in discrete cubical sets.
  The diagonal of this bicubical set is again $X$, which is the homotopy colimit by \cref{diagonal-computes-hocolim}.
\end{proof}

We are now ready to prove \cref{cubical-cat-is-localization}
\begin{proof}[Proof of \cref{cubical-cat-is-localization}]
	We prove the result for weak cotensors; the case of weak tensors proceeds analogously. We will also model $\infty$-groupoids by cubical Kan complexes, which is justified by \cref{kan_comparison}.

	In the homotopy-coherent nerve $\cohnerve \ccat{C}^{\mathrm{f}}$, homotopy equivalences become equivalences.
  Thus, the functor $\ccat{C}_0 \to \cohnerve \ccat{C}^{\mathrm{f}}$ induces the bottom arrow in the (homotopy) commutative diagram:
  \[ \begin{tikzcd}
    {} & {\ccat{C}_0} \ar[rd] \ar[ld] & {} \\
    {\ccat{C}_0[W^{-1}]} \ar[rr, dotted] & {} & {\cohnerve \ccat{C}^{\mathrm{f}}}
  \end{tikzcd} \]
  We show this functor induces an equivalence on mapping space; essential surjectivity is immediate.

  For $Y \in \ccat{C}$, the weak cotensors of $\ccat{C}$ induce a diagram 
  \[ - \pitchfork Y \from \boxcat^\op \to \ccat{C}_0. \]
  Since $[1]^0$ is initial in $\boxcat^\op$ and $[1]^0 \pitchfork Y = Y$, this diagram ascends to a diagram in $Y \slice \ccat{C}_0$.
  We show this diagram is a resolution of $Y$:
  \begin{itemize}
    \item[(R1)] $\boxcat^\op$ is weakly contractible as it has an initial object.
    \item[(R2)] The structure maps $Y \to \cube{n} \pitchfork Y$ are homotopy equivalences by \cref{tensor-cubical-htpy-equiv}.
    \item[(R3)]
    The presheaf $\colim\limits_{[1]^n \in \boxcat^\op} \ccat{C}_0 (-, \cube{n} \pitchfork Y)\from \ccat{C}_0\to \SS$ carries homotopy equivalences to equivalences. To see this, we observe that there are equivalences of cubical sets
\[
	\colim\limits_{[1]^n \in \boxcat^\op} \ccat{C}_0 (X, \cube{n} \pitchfork Y) 
	\cong \colim\limits_{[1]^n \in \boxcat^\op} \big( \ccat{C} (X,Y) \big)_n
	\simeq \ccat{C}(X,Y),
\]
  where the last equivalence uses \cref{cubical-set-geometric-realization}.
    Given a homotopy equivalence $f \from X' \to X$, the naturality in $Z$ from item (2) of \cref{tensors-are-functorial} allows us to identify the map $\colim \ccat{C}_0(X, \cube{n} \pitchfork Y) \to \colim \ccat{C}_0(X', \cube{n} \pitchfork Y)$ with the precomposition map $f^* \from \ccat{C}(X, Y) \to \ccat{C}(X', Y)$, which is a homotopy equivalence.
  \end{itemize}

  Now, we have a diagram of presheaves on $\ccat{C}_0^{\op}$
  \[ \begin{tikzcd}
    {} & {\ccat{C}_0}(-, Y) \ar[rd, "\beta"] \ar[ld, "\alpha"'] & {} \\
    {\ccat{C}_0[W^{-1}]}(-, Y) \ar[rr, dotted] & {} & {\cohnerve \ccat{C}^{\mathrm{f}}}(-, Y).
  \end{tikzcd} \]
By \cref{repr-ex-infty-t-commute}, we also have equivalences
\[
	\cohnerve \ccat{C}^{\mathrm{f}}(-, Y) \simeq \ccat{C}^{\mathrm{f}}(-, Y) \simeq \colim\limits_{[1]^n \in \boxcat^\op} \ccat{C}_0(-, \cube{n} \pitchfork Y).
\]
Under this equivalence, we can identify the map $\beta$  as the map $\phi_{[1]^0}$ in diagram (\ref{eq:phi-square}).
  The map $\alpha$ is identified with the unit of the $\gamma_! \adj \gamma^*$ adjunction instantiated at $\ccat{C}_0(-, Y)$.
  With this, the dotted arrow must be homotopic to the map $\Phi_{[1]^0}$, which is an equivalence by \cref{thm:DMSL}.
\end{proof}

The analog of \cref{cubical-cat-is-localization} for simplicial categories is easily deduced from the cubical version.
\begin{corollary} \label{simplicial-cat-is-localization}
  Let $\scat{C}$ be a simplicial category which admits weak tensors or weak cotensors by $\simp{1}$.
  Then, for any fibrant replacement $\ccat{C} \to \ccat{C}^\mathrm{f}$ of $\ccat{C}$ in the Bergner model structure, the map
  \[ \ccat{C}_0 \ito \cohnerve[\Delta] \ccat{C} \to \cohnerve[\Delta] \ccat{C}^{\mathrm{f}} \]
  is a localization of $\ccat{C}_0$ at homotopy equivalences. 
\end{corollary}
\begin{proof}
  Since the triangulation functor $T \from \cSet \to \sSet$ is a monoidal left adjoint and $T\cube{1} \cong \simp{1}$, the cubical category $U_\bullet \scat{C}$ admits weak (co)tensors by $\cube{1}$ (\cref{ex:change-of-base-tensor}).
  By \cref{cubical-cat-is-localization}, the map $\scat{C}_0 \to \cohnerve U_\bullet \scat{C}^{\mathrm{f}}$ is a localization.
  We recall that $\cohnerve[\Delta] \cong \cohnerve \circ U_\bullet$, which concludes the proof.
\end{proof}
An instance of \cref{simplicial-cat-is-localization} is when $\scat{C}$ is a $(2, 1)$-category.
The \emph{Duskin nerve} of a 2-category $\cohnerve[\mathrm{D}] \cat{C}$ coincides with the homotopy coherent nerve of the base change of $\cat{C}$ along the nerve functor $N:\Cat\to \sSet$ \cite[00BF]{kerodon}. Thus, \cref{simplicial-cat-is-localization} specializes to the following corollary:

\begin{corollary} \label{localization-2-categories}
  Let $\cat{C}$ be a 2-category which admits weak tensors or weak cotensors by the walking arrow category $0 \to 1$. 
  If the mapping categories of $\cat{C}$ are all groupoids then the induced map
  \[ \cat{C}_0 \to \cohnerve[\mathrm{D}] \cat{C} \]
  exhibits $\cohnerve[\mathrm{D}] \cat{C}$ as a localization of $\cat{C}_0$ at equivalences of $\cat{C}$. \qed
\end{corollary}

In our final corollary, we will show that a certain cubical category of graphs considered in discrete homotopy theory is its localization at the class of homotopy equivalences.
This fact was recently used without proof by the last two authors along with Kim \cite{carranza-kapulkin-kim} to resolve one of the open problems posed during the workshop on ``Discrete and Combinatorial Homotopy Theory'' at the American Institute for Mathematics in March 2023.

We begin with a brief introduction to discrete homotopy theory, a homotopy theory of graphs with broad applications in matroid theory, hyperplane arrangements, and data analysis \cite{barcelo-laubenbacher:perspectives,kapulkin-kershaw:data}.

A \emph{graph} is a set $X$ (of vertices) along with a symmetric and reflexive relation (determining the edges between them).
A \emph{graph map} is a function between their underlying sets preserving the relation.
We write $\Graph$ for the category of graphs and graph maps.
An important family of examples of graphs consists of the \emph{interval graphs} $I_n$ with vertices: $0$, $1$, \ldots, $n$, and with relations $i \sim i+1$ for all $i < n$.
There are two canonical maps $l, r \colon I_{n+1} \to I_n$ with $l$ being the (weakly) monotone map repeating $0$ and $r$ being the (weakly) monotone map repeating $n$.

Given two graphs $X$ and $Y$, we define their \emph{box product} as the graph $X \mathop{\square} Y$ whose vertex set is the cartesian product of $X$ and $Y$, and whose edges are given by the following condition: $(x, y) \sim (x', y')$ if and only if either $x \sim x'$ and $y = y'$, or $x = x'$ and $y \sim y'$.
For example, we have $I_1 \mathop{\square} I_1 \iso C_4$.
The box product $\square$ defines a closed symmetric monoidal structure (one of exactly two such structures \cite{kapulkin-kershaw:monoidal}) on the category $\Graph$.
We denote the right adjoint to $X \mathop{\square} -$ by $\Graph^\square(X, -)$.
Concretely, the vertices of $\Graph^\square(X, Y)$ are graph maps $X \to Y$ with $f \sim g$ whenever there exists a graph map $H \colon X \mathop{\square} I_1 \to Y$ such that $H(-, 0) = f$ and $H(-, 1) = g$. 

Given graph maps $f, g \colon X \to Y$, a \emph{homotopy} from $f$ to $g$ consists of a choice of a natural number $n$ and a graph map $H \colon X \mathop{\square} I_n \to Y$ such that $H(-, 0) = f$ and $H(-, n) = g$.
One easily verifies that the relation of being homotopic is an equivalence relation.
A graph map $f \colon X \to Y$ is a homotopy equivalence if there is a graph map $g \colon Y \to X$ with homotopies $gf \Rightarrow \id$ and $fg \Rightarrow \id$.
With these definitions, one can show that $C_3$ and $C_4$ are homotopy equivalent to $I_0$, while the cycles $C_5$, $C_6$, \ldots are pairwise homotopy non-equivalent.
This is in stark contrast to the view of graphs as 1-dimensional CW-complexes, according to which all cycles are homotopy equivalent to $K(\mathbb{Z}, 1)$'s.

We define a family of functors, called the $m$-nerves and denoted $N^\mathbb{G}_m$, associating a cubical set to a graph.
Given a graph $X$, the $k$-cubes of $N^{\mathbb{G}}_m X$ are given by the formula
\[ (N^\mathbb{G}_m X)_k = \Graph(I_m^{\square k}, X)\text{.} \]
It is straightforward to verify that these are indeed cubical sets with connections, symmetries, and reversals, although the last two will not play a role in the remainder.
The maps $l, r \colon I_{n+1} \to I_n$ discussed earlier give rise to natural transformations that we can alternate:
\[ N^\mathbb{G}_1 \xhookrightarrow{l^*} N^\mathbb{G}_2 \xhookrightarrow{r^*} N^\mathbb{G}_3 \xhookrightarrow{l^*} N^\mathbb{G}_4 \xhookrightarrow{r^*} \ldots \]
Let $N^\mathbb{G}$ be the colimit of this sequence.
Explicitly, the $k$-cubes of $N^\mathbb{G}X$ consist of infinite $k$-dimensional grids that stabilize in a suitable sense (see \cite{carranza-kapulkin:cubical-graphs}).
\begin{theorem}[{\cite[Thm.~4.1]{carranza-kapulkin:cubical-graphs}}] \leavevmode
\begin{enumerate}
  \item The functor $N^\mathbb{G}$ takes values in cubical Kan complexes.
  \item The inclusions $l^*, r^* \colon N^\mathbb{G}_m \ito N^\mathbb{G}_{m+1}$ are natural weak equivalences.
\end{enumerate}
\end{theorem}

Furthermore, each $N^\mathbb{G}_m \colon \Graph \to \cSet$ is a lax monoidal functor, as is their colimit $N^\mathbb{G} \colon \Graph \to \cSet$.
Therefore, for any natural number $m$, we can promote the category of graphs to a cubical category by taking its objects to be graphs and defining $\Graph^m(X, Y) = N^\mathbb{G}_m \Graph^\square (X, Y)$.
Similarly, we define $\Graph^\infty$ as the cubical category whose objects are graphs and with $\Graph^\infty(X, Y) = N^\mathbb{G} \Graph^\square(X, Y)$.
All of the cubical categories $\Graph^1, \Graph^2, \ldots, \Graph^\infty$ have the same underlying relative category, namely the category of graphs with weak equivalences given by the homotopy equivalences.
We denote this underlying relative category by $\Graph$.
With this, we can now state our main result:
\begin{corollary} \label{localization-graphs}
  The canonical map
  \[
  \Graph \to \cohnerve \Graph^\infty
  \]
  exhibits $\cohnerve \Graph^\infty$ as the localization of $\Graph$ at the class of homotopy equivalences.
\end{corollary}

\begin{proof}
  The cubical category $\Graph^1$ admits weak tensors (and cotensors) by $\cube{1}$ via $\cube{1} \otimes X = I_1 \mathop{\square} X$ (\cref{ex:self-enriched-tensor,ex:change-of-base-tensor}).
  The inclusion $\Graph^1 \ito \Graph^\infty$ is a weak equivalence of cubical categories, with the latter fibrant.
  The result now follows from \cref{cubical-cat-is-localization}.
\end{proof}

Since it was shown in \cite{carranza-kapulkin-kim} that $\cohnerve \pr{ \Graph^\infty}$ does not have pushouts, we can conclude that the homotopy equivalences of graphs are not part of a model structure on the category of graphs, thus answering in the negative a question from the AIM Workshop on ``Discrete and Combinatorial Homotopy Theory.'

\appendix
\renewcommand{\thesection}{\Alph{section}}

\begin{appendices}

\section{Cubical sets} \label{sec:cubical-sets}

In this section, we recall the definitions and results regarding the homotopy theory of cubical sets which are used elsewhere in the paper.
The most significant of these is the \emph{Grothendieck} model structure on cubical sets, its fibrant objects (i.e.~the \emph{cubical Kan complexes}), and the \emph{triangulation} Quillen equivalence (\cref{thm:cset-model-structure}).
Our primary reference is \cite{doherty-kapulkin-lindsey-sattler} --- other references on the topic include \cite{cisinski:presheaves},\cite{cisinski:elegant}, \cite{jardine:a-beginning}, and \cite{kapulkin-voevodsky}.

We begin by defining the cube category $\Box$.
The objects of $\Box$ are posets of the form $[1]^n = \{0 \leq 1\}^n$ and the maps are generated (inside the category of posets) under composition by the following four special classes:
\begin{itemize}
  \item \emph{faces} $\partial^n_{i,\varepsilon} \colon [1]^{n-1} \to [1]^n$ for $i = 1, \ldots , n$ and $\varepsilon = 0, 1$ given by:
  \[ \partial^n_{i,\varepsilon} (x_1, x_2, \ldots, x_{n-1}) = (x_1, x_2, \ldots, x_{i-1}, \varepsilon, x_i, \ldots, x_{n-1})\text{;}  \]
  \item \emph{degeneracies} $\sigma^n_i \colon [1]^n \to [1]^{n-1}$ for $i = 1, 2, \ldots, n$ given by:
  \[ \sigma^n_i ( x_1, x_2, \ldots, x_n) = (x_1, x_2, \ldots, x_{i-1}, x_{i+1}, \ldots, x_n)\text{;}  \]
  \item \emph{negative connections} $\gamma^n_{i,0} \colon [1]^n \to [1]^{n-1}$ for $i = 1, 2, \ldots, n-1$ given by:
  \[ \gamma^n_{i,0} (x_1, x_2, \ldots, x_n) = (x_1, x_2, \ldots, x_{i-1}, \max\{ x_i , x_{i+1}\}, x_{i+2}, \ldots, x_n) \text{.} \]
  \item \emph{positive connections} $\gamma^n_{i,1} \colon [1]^n \to [1]^{n-1}$ for $i = 1, 2, \ldots, n-1$ given by:
  \[ \gamma^n_{i,1} (x_1, x_2, \ldots, x_n) = (x_1, x_2, \ldots, x_{i-1}, \min\{ x_i , x_{i+1}\}, x_{i+2}, \ldots, x_n) \text{.} \]
\end{itemize}

These maps obey the following \emph{cubical identities}:

\[ \begin{array}{l l}
    \partial_{j, \varepsilon'} \partial_{i, \varepsilon} = \partial_{i+1, \varepsilon} \partial_{j, \varepsilon'} \quad \text{for } j \leq i; & 
    \sigma_j \partial_{i, \varepsilon} = \begin{cases}
        \partial_{i-1, \varepsilon} \sigma_j & \text{for } j < i; \\
        \id                                                       & \text{for } j = i; \\
        \partial_{i, \varepsilon} \sigma_{j-1} & \text{for } j > i;
    \end{cases} \\
    \sigma_i \sigma_j = \sigma_j \sigma_{i+1} \quad \text{for } j \leq i; &
    \gamma_{j,\varepsilon'} \gamma_{i,\varepsilon} = \begin{cases}
    \gamma_{i,\varepsilon} \gamma_{j+1,\varepsilon'} & \text{for } j > i; \\
    \gamma_{i,\varepsilon}\gamma_{i+1,\varepsilon} & \text{for } j = i, \varepsilon' = \varepsilon;
    \end{cases} \\
    \gamma_{j,\varepsilon'} \partial_{i, \varepsilon} = \begin{cases} 
        \partial_{i-1, \varepsilon} \gamma_{j,\varepsilon'}   & \text{for } j < i-1 \text{;} \\
        \id                                                         & \text{for } j = i-1, \, i, \, \varepsilon = \varepsilon' \text{;} \\
        \partial_{j, \varepsilon} \sigma_j         & \text{for } j = i-1, \, i, \, \varepsilon = 1-\varepsilon' \text{;} \\
        \partial_{i, \varepsilon} \gamma_{j-1,\varepsilon'} & \text{for } j > i;
    \end{cases} &
    \sigma_j \gamma_{i,\varepsilon} = \begin{cases}
        \gamma_{i-1,\varepsilon} \sigma_j  & \text{for } j < i \text{;} \\
        \sigma_i \sigma_i           & \text{for } j = i \text{;} \\
        \gamma_{i,\varepsilon} \sigma_{j+1} & \text{for } j > i \text{.} 
    \end{cases}
\end{array} \]

Our convention is to write cubical operators on the right e.g.~given an $n$-cube $x \in X_n$ of a cubical set $X$, we write $x\face{}{1,0}$ for the $\face{}{1,0}$-face of $x$. We will write $\cSet$ for the category of cubical sets.
\begin{definition} \leavevmode
    \begin{enumerate}
        \item For $n \geq 0$, the \emph{combinatorial $n$-cube} $\cube{n}$ is the representable functor $\boxcat(-, [1]^n) \from \boxcat^\op \to \Set$.
        \item The \emph{boundary of the $n$-cube} $\bd \cube{n}$ is the subobject of $\cube{n}$ defined by
        \[ \bd \cube{n} := \bigcup\limits_{\substack{j=1,\dots,n \\ \eta = 0, 1}} \operatorname{Im} \face{}{j,\eta}. \]
\item given $i = 1, \dots, n$ and $\eps = 0, 1$, the $(i, \eps)$-open box  $\dfobox$ is the subobject of $\bd \cube{n}$ defined by
        \[ \dfobox := \bigcup\limits_{(j,\eta) \neq (i,\varepsilon)} \operatorname{Im} \face{}{j,\eta}. \]
    \end{enumerate}
\end{definition}

Cubical sets are related to simplicial sets via the following adjunction.
\begin{definition} \label{def:geom-realization}
        Define a functor $\boxcat \to \sSet$ which sends $[1]^n$ to $(\simp{1})^{n}$.
        Left Kan extension along the Yoneda embedding gives the \emph{triangulation} functor 
        \[ \begin{tikzcd}[column sep = large]
            \boxcat \ar[r, "{[1]^n \mapsto (\simp{1})^n}"] \ar[d] & \sSet \\
            \cSet \ar[ur, "T"']
        \end{tikzcd} \]
        whose right adjoint $U \from \sSet \to \cSet$ is defined by
        \[ (UX)_n := \sSet \left( (\simp{1})^n, X \right) . \] 
\end{definition}
By construction, there is a natural isomorphism $\creali{ X } \cong \sreali{ TX }$ between the cubical geometric realization of $X$ and the simplicial geometric realization of $TX$.

Define a monoidal product $\uvar \gprod \uvar \from \boxcat \times \boxcat \to \boxcat$ on the cube category by $[1]^m \gprod [1]^n = [1]^{m+n}$.
Postcomposing with the Yoneda embedding and left Kan extending gives the \emph{geometric product} of cubical sets.
\[ \begin{tikzcd}
    \boxcat \times \boxcat \ar[r, "\gprod"] \ar[d] & \boxcat \ar[r] & \cSet \\
    \cSet \times \cSet \ar[urr, "\gprod"']
\end{tikzcd} \]
This product is biclosed: for a cubical set $X$, we write $\underline{\cSet} (X, \uvar) \from \cSet \to \cSet$ for the right adjoint to the functor $\uvar \gprod X$.
We do not make use of the left-closed structure.

Moreover, the triangulation functor is strong monoidal, i.e.~it sends geometric products to products.
\begin{proposition} \leavevmode
    \begin{enumerate}
        \item The triangulation functor $T \from \cSet \to \sSet$ is strong monoidal.
        \item Its right adjoint $U \from \sSet \to \cSet$ is lax monoidal.
    \end{enumerate}
\end{proposition}
\begin{proof}
    For (1), it suffices to check on representables, on which it follows by definition as
    \begin{align*}
        T(\cube{m} \gprod \cube{n}) &= T(\cube{m+n}) \\
        &= (\simp{1})^{m+n} \\
        &= (\simp{1})^{m} \times (\simp{1})^n \\
        &= T(\cube{m}) \times T(\cube{n}).
    \end{align*}
    Item (2) follows formally from item (1).
\end{proof}

The following result gives an explicit description of cubes in the geometric product, which we use implicitly throughout \cref{sec:cubical-categories}
\begin{proposition}[{\cite[Prop.~1.24]{doherty-kapulkin-lindsey-sattler}}] \label{thm:gprod_cube}
    Let $X, Y$ be cubical sets.
    \begin{enumerate}
        \item For $k \geq 0$, the $k$-cubes of $X \gprod Y$ consist of all pairs $(x \in X_m, y \in Y_n)$ such that $m + n = k$, subject to the identification $(x\degen{}{m+1}, y) = (x, y\degen{}{1})$.
        \item For $x \in X_m$ and $y \in Y_n$, the faces, degeneracies, and connections of the $(m+n)$-cube $(x, y)$ are computed by
        \begin{align*}
            (x, y)\face{}{i,\varepsilon} &= \begin{cases}
                (x\face{}{i,\varepsilon}, y) & 1 \leq i \leq m \\
                (x, y\face{}{i-m, \varepsilon}) & m+1 \leq i \leq m+n;
            \end{cases} \\
            (x, y)\degen{}{i} &= \begin{cases}
                (x\degen{}{i}, y) & 1 \leq i \leq m+1 \\
                (x, y\degen{}{i-m}) & m+1 \leq i \leq m+n;
            \end{cases} \\
            (x, y)\conn{}{i,\varepsilon} &= \begin{cases}
                (x\conn{}{i,\varepsilon}, y) & 1 \leq i \leq m \\
                (x, y\conn{}{i-m,\varepsilon}) & m+1 \leq i \leq n.
            \end{cases}
        \end{align*} \qed
    \end{enumerate}
\end{proposition}

The category of cubical sets carries the \emph{Grothendieck} model structure (also known as the \emph{Cisinski} model structure), which models the homotopy theory of spaces.
\begin{theorem}[Cisinski, cf.~{\cite[Thms.~1.34 \& 6.26]{doherty-kapulkin-lindsey-sattler}}] \label{thm:cset-model-structure}
    The category of cubical sets carries a model structure where
    \begin{itemize}
        \item cofibrations are monomorphisms;
        \item fibrations are maps with the right lifting property against all open box inclusions, i.e.~the set
        \[ \left\{ \dfobox \ito \cube{n} \mid \begin{array}{l}
            n \geq 1 \\
            i = 1, \dots, n \\
            \varepsilon = 0, 1
        \end{array} \right\} ; \]
        and
        \item weak equivalences are created by triangulation, i.e.~a map $f \from X \to Y$ is a weak equivalence if $Tf \from TX \to TY$ is a weak equivalence in the Kan--Quillen model structure.
    \end{itemize}
    Moreover, the triangulation adjunction
    \[ \begin{tikzcd}
        \adjunct{T}{U}{\cSet}{\sSet}
    \end{tikzcd} \] 
    gives a Quillen equivalence with the Kan--Quillen model structure on simplicial sets. \qed
\end{theorem}
\begin{definition}
	A \emph{cubical Kan complexes} (or simply \emph{Kan complex}) is a cubical set with the right lifting property with respect to all open box inclusions. We write $\cKan$ and $\sKan$ for the categories of cubical Kan complexes and simplicial Kan complexes, respectively.
\end{definition}

\section{Cubical categories} \label{sec:cubical-categories}

The main goal of this section is to establish an analog of the Bergner model structure for cubical categories (\cref{model-str-ccat}), its equivalence with the the Joyal model structure for quasicategories (\cref{str-cohnerve-quillen-equiv}), and an identification of the mapping spaces in a cubical category with the mapping spaces of its associated quasicategory (\cref{repr-ex-infty-t-commute}).

We recall some basic results about enriched categories, of which cubical and simplicial categories are examples.
\begin{definition}
    Let $(\cat{V}, \otimes, 1)$ be a monoidal category.
    A \emph{$\cat{V}$-enriched category} $\cubmc{C}$ (or a \emph{category enriched in $(\cat{V}, \otimes, 1)$}) consists of
    \begin{itemize}
        \item a class of objects $\ob \cubmc{C}$;
        \item for $X, Y \in \ob \cubmc{C}$, a \emph{hom-object} $\cubmc{C}(X, Y) \in \cat{V}$;
        \item for $X, Y, Z \in \ob \cubmc{C}$, a \emph{composition} morphism
        \[ \circ \from \cubmc{C}(Y, Z) \gprod \cubmc{C}(X, Y) \to \cubmc{C}(X, Z); \]
        and
        \item for $X \in \ob \cubmc{C}$, an \emph{identity element} morphism $\eid[X] \from 1 \to \cubmc{C}(X, X)$, 
    \end{itemize}
    such that, for all $X, Y, Z, W \in \ob \cubmc{C}$, the diagrams
    \[ \begin{tikzcd}[column sep = large]
            \cat{C}(Z, W) \otimes \cat{C}(Y, Z) \otimes \cat{C}(X, Y) \ar[r, "{\cat{C}(Z, W) \otimes \mathord{\circ}}"] \ar[d, "{\mathord{\circ} \otimes \cat{C}(X, Y)}"'] & \cat{C}(Z, W) \otimes \cat{C}(X, Z) \ar[d, "\circ"] \\
            \cat{C}(Y, W) \otimes \cat{C}(X, Y) \ar[r, "\circ"] & \cat{C}(X, W)
    \end{tikzcd} \]
    \[ \begin{tikzcd}[column sep = 4.5em]
        \cat{C}(X, Y) \ar[r, "{\cat{C}(X, Y) \otimes \eid[X]}"] \ar[rd, "\id"'] & \cat{C}(X, Y) \otimes \cat{C}(X, X) \ar[d, "\circ"] \\
        {} & \cat{C}(X, Y)
    \end{tikzcd} \qquad \begin{tikzcd}[column sep = 4.5em]
        \cat{C}(Y, Y) \otimes \cat{C}(X, Y) \ar[d, "\circ"'] & \cat{C}(X, Y) \ar[l, "{\eid[Y] \otimes \cat{C}(X, Y)}"'] \ar[ld, "{\id}"] \\
        \cat{C}(X, Y)
    \end{tikzcd} \]
    commute.
\end{definition}
\begin{example}[{\cite[Thms.~I.5.2 \& II.6.4]{eilenberg-kelly:closed-categories}}] \label{ex:closed-cat-is-enriched}
    If $(\cat{V}, \otimes, 1)$ is right closed then, writing $\underline{\cat{V}}(X, -) \from \cat{V} \to \cat{V}$ for the right adjoint to the functor $- \otimes X$, the category $\cat{V}$ ascends to a $\cat{V}$-enriched category where the hom-objects are given by $\underline{\cat{V}}(X, Y)$.
\end{example}
\begin{example}[{\cite[Prop.~II.6.3]{eilenberg-kelly:closed-categories}}] \label{ex:change-of-base}
    If $\Phi \from \cat{W} \to \cat{V}$ is a lax monoidal functor then any $\cat{W}$-enriched category $\cubmc{D}$ gives rise to a $\cat{V}$-category $\Phi_\bullet \cubmc{D}$ with the same objects.
    The hom-object from $X$ to $Y$ is given by $\Phi(\cubmc{D}(X, Y))$.
\end{example}
\begin{definition}
    Let $(\cat{V}, \otimes, 1)$ be a monoidal category.
    A \emph{$\cat{V}$-enriched functor} $F \from \cubmc{C} \to \cubmc{D}$ between $\cat{V}$-enriched categories consists of
    \begin{itemize}
        \item an assignment on objects $\ob \cubmc{C} \to \ob \cubmc{D}$;
        \item for $X, Y \in \ob \cubmc{C}$, a morphism
        \[ F_{X, Y} \from \cubmc{C}(X, Y) \to \cubmc{D}(FX, FY) \]
        in $\cat{V}$
    \end{itemize}
    such that, for $X, Y, Z \in \ob \cubmc{C}$, the diagrams
    \[ \begin{tikzcd}
        \cubmc{C}(Y, Z) \gprod \cubmc{C}(X, Y) \ar[r, "\circ"] \ar[d, "{F_{Y, Z} \gprod F_{X, Y}}"'] & \cubmc{C}(X, Z) \ar[d, "F_{X, Z}"] \\
        \cubmc{D}(FY, FZ) \gprod \cubmc{D}(FX, FY) \ar[r, "\circ"] & \cubmc{D}(FX, FZ)
    \end{tikzcd} \qquad \begin{tikzcd}
        \cube{0} \ar[r, "{\eid[X]}"] \ar[rd, "{\eid[FX]}"'] & \cubmc{C}(X, X) \ar[d, "F_{X, X}"] \\
        {} & \cubmc{D}(FX, FX)
    \end{tikzcd} \]
    commute.
\end{definition}
We write $\Cat_{\cat{V}}$ for the category of $\cat{V}$-enriched categories and $\cat{V}$-enriched functors.
\begin{example}[{\cite[Thms.~I.6.6 \& II.6.4]{eilenberg-kelly:closed-categories}}] \label{ex:monoidal-functor-enriched-ascend}
	Any lax monoidal functor $\Phi \from \cat{W} \to \cat{V}$ between right closed monoidal categories ascends to a $\cat{V}$-enriched functor
    \[ \overline{\Phi} \from \Phi_\bullet \cat{W} \to \cat{V} \]
    which sends an object $x \in \Phi_\bullet \cat{W}$ to $\Phi x \in \cat{V}$.
    The action $\overline{\Phi}_{X, Y} \from \Phi \cat{W}(X, Y) \to \cat{V}(\Phi X, \Phi Y)$ on hom-objects is adjoint transpose to the map
    \[ \Phi (\cat{W}(X, Y)) \otimes \Phi X \to \Phi (\cat{W}(X, Y) \otimes X) \xrightarrow{\Phi (\mathrm{eval}_Y)} \Phi Y. \]
\end{example}
\begin{example}[{\cite[Prop.~II.6.3, Thm.~II.6.5]{eilenberg-kelly:closed-categories}}] \label{ex:change-of-base-functorial}
    The construction given in \cref{ex:change-of-base} is functorial, i.e.~any lax monoidal functor $\Phi \from \cat{W} \to \cat{V}$ induces a functor
    \[ \Phi_\bullet \from \Cat_{\cat{W}} \to \Cat_{\cat{V}}. \]
    In particular, any $\cat{W}$-enriched functor $\cubmc{C} \to \cubmc{D}$ gives rise to a $\cat{V}$-enriched functor $\Phi_\bullet \cubmc{C} \to \Phi_\bullet \cubmc{D}$.

    Given another lax monoidal functor $\Psi \from \cat{U} \to \cat{W}$, the functor $(\Phi \Psi)_\bullet$ is equal to $\Phi_\bullet \Psi_\bullet$.
\end{example}
\begin{definition} \label{def:cubical-nat-transf}
    Given $\cat{V}$-enriched functors $F, G \from \cubmc{C} \to \cubmc{D}$, a \emph{$\cat{V}$-enriched natural transformation} $\alpha$ from $F$ to $G$ consists of, for each object $X \in \cubmc{C}$, a morphism $\alpha_X \from 1 \to \cubmc{D}(FX, GX)$ in $\cat{V}$ such that, for $X, Y \in \cubmc{C}$, the square
    \[ \begin{tikzcd}[sep = large]
        \cubmc{C}(X, Y) \ar[rr, "F_{X, Y}"] \ar[dd, "G_{X, Y}"'] & {} & \cubmc{D}(FX, FY) \ar[d, "{\alpha_Y \gprod \id}"] \\
        {} & {} & \cubmc{D}(FY, GY) \gprod \cubmc{D}(FX, FY) \ar[d, "\circ"] \\
        \cubmc{D}(GX, GY) \ar[r, "{\id \gprod \alpha_X}"] & {\cubmc{D}(GX, GY) \gprod \cubmc{D}(FX, GX)} \ar[r, "\circ"] & \cubmc{D}(FX, GY)
    \end{tikzcd} \]
    commutes.
\end{definition}
\begin{example}[{\cite[Prop.~II.6.3, Thm.~II.6.5]{eilenberg-kelly:closed-categories}}] \label{ex:monoidal-nat-transf-enriched}
	Any monoidal natural transformation $\eta$ between lax monoidal functors $\Phi, \Psi \from \cat{W} \to \cat{V}$ induces a natural transformation $\eta_\bullet$ from $\Phi_\bullet$ to $\Psi_\bullet$. 
\end{example}
\subsection*{Cubical and simplicial Kan complexes}

Our interest in enriched categories lies in studying cubical and simplicial categories, which we now define.
\begin{definition} \leavevmode
    \begin{itemize}
        \item A \emph{cubical category} is a category enriched in $(\cSet, \gprod, \cube{0})$.
        A \emph{cubical functor} is a $\cSet$-enriched functor between cubical categories.
        \item A \emph{simplicial category} is a category enriched in $(\sSet, \times, \simp{0})$.
        A \emph{simplicial functor} is an $\sSet$-enriched functor between simplicial categories.
    \end{itemize}
\end{definition}
Let $\cCat$ denote the category of cubical categories, and $\sCat$ denote the category of simplicial categories.

In a cubical or simplicial category, we refer to the hom-object from $X$ to $Y$ as the \emph{mapping space} from $X$ to $Y$.
Following \cref{ex:closed-cat-is-enriched}, we write $\cSet$ for the cubical category of cubical sets with mapping spaces given by $\underline{\cSet}(X, Y)$.
Likewise, $\sSet$ denotes the simplicial category of simplicial sets.
We write $\cKan$ for the full cubical subcategory of $\cSet$ spanned by cubical Kan complexes, and $\sKan$ for the full simplicial subcategory of $\sSet$ spanned by simplicial Kan complexes.

Following \cref{ex:change-of-base-functorial}, the triangulation adjunction $\adjunct{T}{U}{\cSet}{\sSet}$ induces an adjunction
\[ \adjunct{T_\bullet}{U_\bullet}{\cCat}{\sCat}. \]

Simplicial and cubical categories which are locally Kan give examples of quasi-categories via their \emph{homotopy-coherent nerve}.
This construction is defined via pre-composition with a cosimplicial object (in $\cCat$ or $\sCat$, respectively), which we shall define.
\begin{remark}
    Our primary reference for results regarding the cubical homotopy-coherent nerve is \cite{kapulkin-voevodsky}.
    Although this reference is written using the categorical product on cubical sets as the base for enrichment (as opposed to the geometric product), any results used apply to both products.
\end{remark}


For $n \geq 0$, define a cubical category $\str [n]$ as follows:
\begin{itemize}
    \item the set of objects of $\str [n]$ is $\{ 0, \dots, n \}$;
    \item for $i, j \in \{ 0, \dots, n \}$, the mapping space from $i$ to $j$ is defined by
    \[ \str[n](i, j) := \begin{cases}
        \cube{j - i - 1} & i < j \\
        \{ \id[i] \} & i = j \\
        \varnothing & i > j.
    \end{cases} \]
    \item for $i, j, k \in \{ 0, \dots, n \}$, the composition map when $i < j < k$ is given by
	    \[  \cube{k - j - 1} \gprod \cube{j - i - 1} \cong \cube{k-i-2} \xrightarrow{\face{}{j,1}} \cube{k-i-1}. \]
    If $i = j$ or $j = k$ then composition is defined to be the natural isomorphisms
    \[ \str[n](j, k) \gprod \{ \id[j] \} \xrightarrow{\cong} \str[n](i, k) \xleftarrow{\cong} \{ \id[j] \} \gprod \str[n](i, j), \]
    respectively.
    Otherwise, the composition map is trivial as the domain becomes the empty cubical set.
\end{itemize}
We omit definitions of the simplicial operators of $\str [n]$, which may be found in \cite[Lem.~2.2]{kapulkin-voevodsky}.
Extension by colimits yields a functor $\str \from \sSet \to \cCat$.
The \emph{cubical homotopy-coherent nerve} $\cohnerve \from \cCat \to \sSet$ is its right adjoint, which is defined by
\[ (\cohnerve \cubmc{C})_n := \cCat(\str[n], \cubmc{C}). \]
The simplicial homotopy-coherent nerve $\cohnerve[\Delta] \from \sCat \to \sSet$ is defined by
\[ (\cohnerve[\Delta] \cubmc{D})_n := \cCat(T_\bullet \str[n], \cubmc{D}). \]
In particular, there is a natural isomorphism $\cohnerve[\Delta] \cong \cohnerve \circ U_\bullet$ (\cite[Prop.~2.8]{kapulkin-voevodsky}).

\Cref{explicit-description-of-cubical-str} provides an explicit description of the functor $\str \from \sSet \to \cCat$.
Similar descriptions may be found in \cite[Rem.~5.2]{rivera-zeinalian:cubical-rigidification}, and in the simplicial case in \cite[Thm.~16.4.1]{riehl:categorical-homotopy-theory} and \cite[Cor.~4.4]{dugger-spivak:rigidification} for the functor $\str_{\Delta} \from \sSet \to \sCat$.
\begin{remark} \label{explicit-description-of-cubical-str}
    For a simplicial set $X$, the objects of $\str X$ are the vertices of $X$. Given vertices $x, y$ of $X$, the $k$-cubes of the mapping space $\str X(x, y)$ are finite tuples of pairs
\[ \left( (s_n, f_n), \dots, (s_1, f_1) \right) \]
where
\begin{itemize}
    \item each $s_i$ is a non-degenerate $(m_i)$-simplex of $X$ for some $m_i \geq 1$ such that
    \begin{itemize}
        \item $\restr{s_1}{\{ 0 \}} = x$ and $\restr{s_n}{\{ m_n \}} = y$; and
        \item for $i = 1, \dots, n-1$, $\restr{s_i}{\{ m_i \}} = \restr{s_{i+1}}{\{ 0 \}}$;
    \end{itemize}
    \item each $f_i$ is a $(d_i)$-cube of the mapping space $\str [m_i](0, m_i) \cong \cube{m_i - 1}$ for some $d_i$, where $d_1 + \dots + d_n = k$,
\end{itemize}
subject to the equivalence relation generated by
\begin{itemize}
    \item for $i = 1, \dots, n$ and $j = 1, \dots, m_i$, the relation
    \begin{align*}
        &\left(  (s_n, f_n), \dots, (s_i \face{}{j}, f_i), \dots, (s_1, f_1) \right) \sim \left( (s_n, f_n), \dots, (s_i, \str(\face{}{j})_{0, m_i - 1} \circ f_i), \dots, (s_1, f_1) \right);
    \end{align*}
    and
    \item for $i = 1, \dots, n-1$, the relation
    \begin{align*}
        &\left(  (s_n, f_n), \dots, (s_{i+1}, f_{i+1}), (s_i, f_i \degen{}{m_i + 1}), \dots, (s_1, f_1) \right) \\
        &\quad \sim \left( (s_n, f_n), \dots, (s_{i+1}, f_{i+1} \degen{}{1}), (s_i, f_i), \dots, (s_1, f_1) \right).
    \end{align*} 
\end{itemize}
If $x = y$ then we identify the empty tuple with the 0-cube $\id[x] \in \str X(x, x)$.
Otherwise, we consider all tuples to be non-empty.
Composition is given by concatenation of tuples, hence if each $d_i = 0$ for all $i$ (that is, each $f_i$ is a 0-cube) then
\[ [\left( (s_n, f_n), \dots, (s_1, f_1) \right)] = [(s_n, f_n)] \circ \dots \circ [(s_1, f_1)] \]
in $\pi_0 (\str(X)(x, y))$.

If $X$ is the homotopy-coherent nerve of a cubical category $\cubmc{C}$, an $m_i$-simplex $s_i$ of $X = \cohnerve \cubmc{C}$ is a cubical functor $s_i \from \str [m_i] \to \cubmc{C}$.
With this, the counit map $(\varepsilon_{\cubmc{C}})_{x, y} \from \str \cohnerve \cubmc{C}(x, y) \to \cubmc{C}(x, y)$ has an explicit formula
\[ \left( (s_n, f_n), \dots, (s_1, f_1) \right) \mapsto (s_n f_n) \circ \dots \circ (s_1 f_1), \]
where $s_i f_i$ denotes the composite
\[ \cube{d_i} \xrightarrow{f_i} \str [m_i](0, m_i) \xrightarrow{(s_i)_{0, m_i}} \cubmc{C}(s_i(0), s_i(m_i)). \]
\end{remark}

We move towards developing the homotopy theory of cubical categories, and its connection to the homotopy theory of simplicial categories.
\begin{definition}
    \leavevmode
    \begin{enumerate}
        \item A cubical category $\cubmc{C}$ is \emph{locally Kan} if for all $X, Y \in \ob \cubmc{C}$, the mapping space from $X$ to $Y$ is a cubical Kan complex.
        \item A simplicial category $\cubmc{D}$ is \emph{locally Kan} if for all $X, Y \in \ob \cubmc{C}$, the mapping space from $X$ to $Y$ is a simplicial Kan complex.
    \end{enumerate}
\end{definition}
\begin{proposition} \label{locally-kan-to-qcat}
    \leavevmode
    \begin{enumerate}
        \item If $\cubmc{C}$ is a locally Kan cubical category then $\cohnerve \cubmc{C}$ is a quasicategory.
        \item If $\cubmc{D}$ is a locally Kan simplicial category then $\cohnerve[\Delta] \cubmc{D}$ is a quasicategory.
    \end{enumerate}
\end{proposition}
\begin{proof}
    Item (1) is \cite[Thm.~2.6]{kapulkin-voevodsky}.
    Item (2) is \cite[Prop.~1.1.5.10]{lurie:htt}.
\end{proof}
Locally Kan simplicial categories are the fibrant objects in the \emph{Bergner model structure} on simplicial categories \cite{bergner}.
The simplicial homotopy-coherent nerve is the right adjoint in a Quillen equivalence between the Bergner model structure and the Joyal model structure (\cite[Thm.~2.2.5.1]{lurie:htt}).
We move towards defining an analogous model structure on cubical categories which is Quillen equivalent to these (which we refer to as the \emph{cubical Bergner model structure}).
This will require some auxiliary definitions and results.

If $X$ is a cubical set, the set $\pi_0(X)$ of its \emph{path-components} is the quotient of the set of vertices of $X$ with respect to the equivalence relation generated by the relation of being connected by 1-cubes.
Equivalently, $\pi_0(X)$ is the colimit of $X$, regarded as a cubical object in $\Set$. The right adjoint of $\pi_0 \from \cSet \to \Set$ is denoted by $\Sk_0$.
The adjunction $\adjunct{\pi_0}{\Sk_0}{\cSet}{\Set}$ induces an adjunction
\[ \adjunct{(\pi_0)_\bullet}{(\Sk_0)_\bullet}{\cCat}{\Cat} \]
We write $\Ho \from \cCat \to \Cat$ for the functor $(\pi_0)_\bullet$ and refer to $\Ho \cubmc{C}$ as the \emph{homotopy category} of $\cubmc{C}$.
Its right adjoint $(\Sk_0)_\bullet$ is a full subcategory inclusion, hence any ordinary category may be viewed as a cubical category with discrete mapping spaces.
For an ordinary category $\cat{C}$, we often write the cubical category $(\Sk_0)_\bullet (\cat{C})$ as simply $\cat{C}$.

\begin{proposition} \label{ho-cat-equiv}
    For a cubical category $\cubmc{C}$, there is an isomorphism of categories
    \[ \Ho (\cubmc{C}) \cong \Ho (\cohnerve \cubmc{C}) \]
    natural in $\cubmc{C}$.
\end{proposition}
\begin{proof}
	Recall that, for a simplicial set $X$, its homotopy category $\Ho X$
is defined as the value of the left adjoint to the nerve functor $\nerve \from \Cat \to \sSet$.
Explicitly, it is the category generated by the 1-simplices $X$, subject
to the following relations: 
\begin{itemize}
    \item whenever there is a $2$-simplex in X 
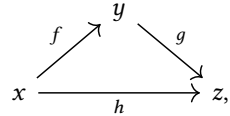
    \[\begin{tikzcd}
        & y \\
        x && z,
        \arrow["g", from=1-2, to=2-3]
        \arrow["f", from=2-1, to=1-2]
        \arrow["h"', from=2-1, to=2-3]
    \end{tikzcd}\] 
    the formal composite $(g, f)$ of $g$ with $f$ is identified with $h$;
    \item every degenerate 1-simplex is identified with an identity morphism.
\end{itemize}
We will denote the equivalence class of a 1-simplex $f \in X_{1}$ in $\Ho X$ by $[f]$. 

Now by definition, a $2$-simplex of $\cohnerve \cubmc{C}$ is a cubical
functor $\str[2] \to \cubmc C$. Such a cubical functor corresponds
the data of:
\begin{itemize}
\item objects $X_{0},X_{1},X_{2} \in \cubmc{C}$;
\item morphisms $f_{ij} \from X_{i} \to X_{j}$ in $\cubmc C_{0}$; and
\item a 1-cube $\alpha \from f_{02} \to f_{12} \circ f_{01}$ in the mapping space $\cubmc C\pr{X_{0},X_{2}}$.
\end{itemize}
We depict this data as a diagram 
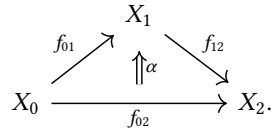
\[\begin{tikzcd}
	& {X_1} \\
	{X_0} && {X_2.}
	\arrow["{f_{12}}", from=1-2, to=2-3]
	\arrow["{f_{01}}", from=2-1, to=1-2]
	\arrow[""{name=0, anchor=center, inner sep=0}, "{f_{02}}"', from=2-1, to=2-3]
	\arrow["\alpha"', shorten=1.4ex, yshift=0.1ex, Rightarrow, from=0, to=1-2]
\end{tikzcd}\]
This means that we can define functors
\[
\Phi \from \Ho {\cubmc C} \to \Ho ({\cohnerve\cubmc C}) \qquad \Psi \from \Ho ({\cohnerve\cubmc C}) \to \Ho {\cubmc C}
\]
by setting $\Phi\pr{ [f] }=[f]$ and $\Psi \big( [f] \big) = [f]$. 
Note that we identify a morphism in $\cubmc{C}$ with the corresponding 1-simplex of $\cohnerve\cubmc C$, and note also that we define $\Psi$ by specifying its actions on the generating arrows.

To show that $\Phi$ is well-defined on the level of morphisms, it suffices to show
that, given a pair of objects $X,Y\in\cubmc C$ and an edge $\alpha \from\square^{1}\to\cubmc C\pr{X,Y}$
depicted as $f\to g$, we have $[f]=[g]$ in $\Ho\pr{\cohnerve \cubmc C}$.
This data gives rise to a $2$-simplex $({\id}_{X},f, g, \alpha) \from \str [2] \to \cubmc{C}$
of $\cohnerve\cubmc C$, so we have $[f]=[g]$ in $\Ho\pr{\cohnerve \cubmc C}$,
as required. 
A similar argument shows that $\Psi$ commutes with compositions and carries identity morphisms to identity maps. 
Thus $\Phi$ and $\Psi$ are well-defined functors which are inverses of each other, yielding the desired isomorphism.
\end{proof}
\begin{remark}
	This is a cubical analog of the result stated in \cite[Warn.~1.2.3.3]{lurie:htt}.
	Here, we are able to prove it without the additional assumption that $\cubmc{C}$ is locally Kan.
\end{remark}

\Cref{ho-cat-equiv} has the following corollary.
We write $E^1$ for the nerve of the thin groupoid with two objects 0,1.
\begin{corollary} \label{invertible-arrow-in-ho}
    Given an arrow $f \from [1] \to \cubmc{C}$ in a locally Kan cubical category $\cubmc{C}$, the following are equivalent:
    \begin{enumerate}
        \item $f$ becomes an isomorphism in $\Ho \cubmc{C}$;
	\item $f$ lifts to a map $\str(E^1) \to \cubmc{C}$;
        \item the adjoint transpose $\tilde{f} \from \simp{1} \to \cohnerve \cubmc{C}$ of $f$ extends to a map $E^1 \to \cohnerve \cubmc{C}$.
    \end{enumerate}
\end{corollary}
\begin{proof}
    The homotopy-coherent nerve of $\cubmc{C}$ is a quasicategory by \cref{locally-kan-to-qcat}.
    Thus, we have implications (1) $\implies$ (3) $\implies$ (2) $\implies$ (1), where the first implication follows from \cref{ho-cat-equiv}.
\end{proof}

Define a functor $\cSusp \from \cSet \to \cCat$ which sends a cubical set $X$ to the cubical category $\cSusp X$ with two objects $0, 1$ and whose mapping spaces are
\[ \begin{array}{r@{ \ := \ }l  r@{ \ := \ }l}
    \cSusp X(0, 1) & X & \cSusp X(1, 0) & \varnothing \\
    \cSusp X(0, 0) & \{ \id[0] \} & \cSusp X (1, 1) & \{ \id[1] \}.
\end{array} \]

\begin{theorem} \label{model-str-ccat}
    The category of cubical categories carries a model structure where:
    \begin{itemize}
        \item cofibrations are the saturation of the set
        \[ \{ \varnothing \ito [0] \} \cup \{ \cSusp \bd \cube{n} \ito \cSusp \cube{n} \mid n \geq 0 \}; \]
        \item fibrant objects are locally Kan cubical categories; fibrations between fibrant objects are maps with the right lifting property against the set
        \[ \{ [0] \ito \str (E^1) \} \cup \{ \cSusp \dfobox \ito \cSusp \cube{n} \mid n \geq 1, \ i = 1, \dots, n \ \varepsilon = 0, 1 \}; \]
        \item weak equivalences are DK-equivalences, i.e.~cubical functors $F \from \cubmc{C} \to \cubmc{D}$ such that
        \begin{enumerate}
            \item $\Ho F \from \Ho \cubmc{C} \to \Ho \cubmc{D}$ is an equivalence of categories; and
            \item for all $x, y \in \cubmc{C}$, the cubical map
            \[ {F}_{x, y} \from \cubmc{C}(x, y) \to \cubmc{D}(x, y) \]
            is a weak equivalence in the Grothendieck model structure.
        \end{enumerate}
    \end{itemize}
    Moreover, the triangulation adjunction induces a Quillen equivalence
    \[ \adjunct{T_\bullet}{U_\bullet}{\cCat}{\Cat_{\Delta}} \]
    with the Bergner model structure on simplicial categories.
\end{theorem}
\begin{proof}
    For existence of the model structure as well as the description of the cofibrations and weak equivalences, we apply \cite[Prop.~A.3.2.4]{lurie:htt}.
    The Quillen equivalence follows from \cite[Rem.~A.3.2.6]{lurie:htt}.

    For the description of fibrant objects, we apply \cite[Prop.~2.3]{stanculescu}, where $I$ is the set of generating cofibrations in the theorem statement and $J$ is the set
    \[ J = \{ [0] \ito \str (E^1) \} \cup \{ \cSusp \dfobox \ito \cSusp \cube{n} \mid n \geq 1, \ i = 1, \dots, n, \ \varepsilon = 0, 1 \}. \]
    One verifies that a locally Kan cubical category is exactly a cubical category with the right lifting property with respect to $J$.
    We show that a weak equivalence between locally Kan cubical categories that has the right lifting property with respect to $J$ is an acyclic fibration; the remaining hypotheses of \cite[Prop.~2.3]{stanculescu} are immediate.

    To this end, fix such a map $F \from \cubmc{C} \xrightarrow{\sim} \cubmc{D}$.
    It suffices to show this map has the right lifting property with respect to the set of generating cofibrations
    \[ \{ \varnothing \ito [0] \} \cup \{ \cSusp \bd \cube{n} \ito \cSusp \cube{n} \mid n \geq 0 \}. \]
    The maps $\cSusp \bd \cube{n} \ito \cSusp \cube{n}$ are immediate, so it remains to show $F$ has the right lifting property against the map $\varnothing \ito [0]$.
    Given a square
    \[ \begin{tikzcd}
        \varnothing \ar[r] \ar[d] & \cubmc{C} \ar[d, "F"] \\
        {[0]} \ar[r, "y"] & \cubmc{D}
    \end{tikzcd} \]
    there exists $x \in \cubmc{C}$ and an isomorphism $f \from E[1] \to \Ho \cubmc{D}$ from $\Ho F(x)$ to $y$.
    As $\cubmc{D}$ is locally Kan, applying \cref{invertible-arrow-in-ho} gives a cubical functor $\overline{f} \from \str (E^1) \to \cubmc{D}$ such that $\Ho \overline{f} = f$ .
    The square
    \[ \begin{tikzcd}
        {[0]} \ar[r, "x"] \ar[d, hook] & \cubmc{C} \ar[d, "F"] \\
        \str (E^1) \ar[r, "\overline{f}"] & \cubmc{D}
    \end{tikzcd} \]
    admits a lift $g \from \str (E^1) \to \cubmc{C}$ by assumption.
    The object $g(1)$ is a lift of the starting square.
\end{proof}
\begin{proposition} \label{str-cohnerve-quillen-equiv}
    The adjunction
    \[ \adjunct{\str}{\cohnerve}{\sSet}{\cCat} \]
    is a Quillen equivalence.
\end{proposition}
\begin{proof}
    We first show that $\str \adj \cohnerve$ is a Quillen adjunction.
    By \cite[Prop.~7.15]{joyal-tierney:qcat-vs-segal}, it suffices to show $\str \from \sSet \to \cCat$ sends monomorphisms to cofibrations and $\cohnerve \from \cCat \to \sSet$ sends fibrations between fibrant objects to fibrations.

    As $\str$ is a left adjoint, it suffices to show it takes boundary inclusions to cofibrations.
    This follows as the square
    \[ \begin{tikzcd}
        \cSusp \bd \cube{n} \ar[r, "0 \mapsto 0", "1 \mapsto n"'] \ar[d, hook] & \str (\bd \simp{n}) \ar[d, hook] \\
        \cSusp \cube{n} \ar[r, "0 \mapsto 0", "1 \mapsto n"'] & \str [n] 
    \end{tikzcd} \]
    is a pushout for all $n \geq 1$.
    The inclusion $\str(\bd \simp{0}) \ito \str [0]$ is a cofibration by definition.

    The functor $\cohnerve$ sends locally Kan quasicategories to quasicategories by \cref{locally-kan-to-qcat}.
    One verifies that it sends fibrations between fibrant objects to inner isofibrations between quasicategories, thus $\cohnerve$ preserves fibrations between fibrant objects.

    In the commutative triangle of adjunctions
    \[ \begin{tikzcd}
        \sSet \ar[rr, "{\str}", yshift=0.6ex] \ar[rd, "{\str}", xshift=0.3ex, yshift=0.3ex] & {} & \Cat_{\Delta} \ar[ll, yshift=-0.6ex, "{\cohnerve[\Delta]}"] \ar[ld, "{U_\bullet}", xshift=0.5ex, yshift=-0.5ex] \\
        {} & \cCat \ar[ur, "{T_\bullet}", xshift=-0.3ex, yshift=0.3ex] \ar[ul, "\cohnerve", xshift=-0.5ex, yshift=-0.5ex] & {}
    \end{tikzcd} \]
    the top adjunction is a Quillen equivalence by \cite[Thm.~2.2.5.1]{lurie:htt}.
    The bottom right adjunction is a Quillen equivalence by \cref{model-str-ccat}.
    Thus, $\str \adj \cohnerve$ is a Quillen equivalence by 2-out-of-3.
\end{proof}
Our final goal is to construct an equivalence of quasicategories
$\cohnerve[\Delta] \sKan \to \cohnerve \cKan$ and use it to identify mapping spaces of cubical categories with their quasicategorical counterparts.
The functor $U:\sSet \to\CS$ is lax monoidal, so it induces a cubical
functor
\[
U_{\bullet}\sSet \to\CS,
\]
which restricts to a cubical functor $U_{\bullet}\sKan\to\cKan$.
Applying the cubical homotopy coherent nerve yields a functor
$N_{\Delta}\pr{\sKan}\to \cohnerve\pr{\cKan}$.
\begin{theorem}\label{kan_comparison}
The functor $\cohnerve[\Delta]\sKan\to \cohnerve\cKan$ is a
categorical equivalence.
\end{theorem}

\begin{proof}
It suffices to show that the functor $\Psi:U_{\bullet}\sKan\to\cKan$
is a Dwyer--Kan equivalence. The Quillen equivalence $T\dashv U$
implies that $\Ho\pr{\Psi}$ is essentially surjective. Therefore,
we are reduced to showing that the map
\[
\Psi_{X,Y}:U\pr{\sKan\pr{X,Y}}\to\cKan\pr{UX,UY}
\]
is a homotopy equivalence for every pair of objects $X,Y\in\sKan$.
We claim that this map is equal to the composite
\[
\theta:U\pr{\sKan\pr{X,Y}}\xrightarrow{U\pr{\varepsilon_{X}^{*}}}U\pr{\sKan\pr{TU\pr X,Y}}\cong\cKan\pr{UX,UY},
\]
where $\varepsilon_{X}^{*}$ denotes the precomposition along the
counit map $\varepsilon_{X}:TU\pr X\to X$. Since $X$ is a Kan complex,
the map $\varepsilon_{X}$ is a weak homotopy equivalence, and so this
will show that $U\pr{\varepsilon_{X}^{*}}$ is a homotopy equivalence.

To identify $\Psi_{X,Y}$ with $\theta$, recall that the structure
map $\alpha:UA\otimes UB\to U\pr{A\times B}$ of the lax monoidal
functor $U$ is given by the composite
\[
UA\otimes UB\xrightarrow{\eta_{UA\otimes UB}}UT\pr{UA\otimes UB}\cong U\pr{\pr{TUA}\times\pr{TUB}}\xrightarrow{U\pr{\varepsilon_{A}\times\varepsilon_{B}}}U\pr{A\times B},
\]
where $\eta$ denotes the unit and the counit of the adjunction $T\dashv U$,
and the middle isomorphism comes from the strong monoidality of $T$.
By definition, the map $\Psi_{X,Y}$ is the transpose of the composite
\[
\Psi'_{X,Y}:U\sKan\pr{X,Y}\otimes UX\xrightarrow{\alpha}U\pr{\sKan\pr{X,Y}\times X}\xrightarrow{U\pr{\mathrm{eval}_{X}}}UY.
\]
Using the above description of $\alpha$, we find that the transpose
of $\Psi'_{X,Y}$ is equal to the outer right composite from
the top left to $Y$ in the following diagram:
\[\begin{tikzcd}[scale cd = .65]
	{TU(\sKan(X,Y)\otimes UX)} & {TUT(U\sKan(X,Y)\otimes UX)} & {TU(TU\sKan(X,Y)\times TU(X))} & {TU(\sKan(X,Y)\times X)} & TUY \\
	{TU(\sKan(X,Y)\otimes UX)} & {T(U\sKan(X,Y)\otimes UX)} & {TU\sKan(X,Y)\times TU(X)} & {\sKan(X,Y)\times X} & Y \\
	& {T(U\sKan(X,Y)\otimes UX)} & {TU\sKan(X,Y)\times TU(X)} & {\sKan(X,Y)\times TUX} & {\sKan(TUX,Y)\times TUX} \\
	& {T(U\sKan(TUX,Y)\otimes UX)} && {\sKan(TUX,Y)\times TUX}
	\arrow[""{name=0, anchor=center, inner sep=0}, "{T\eta}", from=1-1, to=1-2]
	\arrow[equal, from=1-1, to=2-1]
	\arrow["\cong", from=1-2, to=1-3]
	\arrow["{\varepsilon T}", from=1-2, to=2-2]
	\arrow["{TU(\varepsilon\times\varepsilon)}", from=1-3, to=1-4]
	\arrow["\varepsilon", from=1-3, to=2-3]
	\arrow["{TU(\mathrm{eval})}", from=1-4, to=1-5]
	\arrow["\varepsilon", from=1-4, to=2-4]
	\arrow["\varepsilon", from=1-5, to=2-5]
	\arrow[""{name=1, anchor=center, inner sep=0}, equal, from=2-1, to=2-2]
	\arrow["\cong", from=2-2, to=2-3]
	\arrow[equal, from=2-2, to=3-2]
	\arrow[""{name=2, anchor=center, inner sep=0}, "{\varepsilon\times \varepsilon}", from=2-3, to=2-4]
	\arrow[equal, from=2-3, to=3-3]
	\arrow[""{name=3, anchor=center, inner sep=0}, "{\mathrm{eval}}", from=2-4, to=2-5]
	\arrow["\cong", from=3-2, to=3-3]
	\arrow["{T(U\varepsilon^*\otimes \mathrm{id})}"', from=3-2, to=4-2]
	\arrow[""{name=4, anchor=center, inner sep=0}, "{\varepsilon\times \mathrm{id}}"', from=3-3, to=3-4]
	\arrow["{TU\varepsilon^*\times \mathrm{id}}"', from=3-3, to=4-4]
	\arrow["{\mathrm{id}\times \varepsilon}"', from=3-4, to=2-4]
	\arrow[""{name=5, anchor=center, inner sep=0}, "{\varepsilon^*\times \mathrm{id}}"', from=3-4, to=3-5]
	\arrow["{\mathrm{eval}}"', from=3-5, to=2-5]
	\arrow["\cong"', from=4-2, to=4-4]
	\arrow["{\varepsilon\times \mathrm{id}}"', from=4-4, to=3-5]
	\arrow["{(\text{a})}"{description}, draw=none, from=0, to=1]
	\arrow["{(\text{c})}"{description}, draw=none, from=2, to=4]
	\arrow["{(\text{b})}"{description}, draw=none, from=3, to=5]
\end{tikzcd}\]

Square (a) commutes by the triangle identity, square (b) commutes
by the definition of the evaluation map, and square (c) obviously
commutes. All the other squares commute by the naturality of the counit and the monoidal structure maps.
So the transpose of $\Psi'_{X,Y}$ is equal to the outer left composite,
which is exactly the transpose of $\theta$. Thus $\Psi_{X,Y}$ and
$\theta$ are equal, as claimed.
\end{proof}
We now use \cref{kan_comparison} to identify mapping spaces of cubical
categories with those of the cubical homotopy-coherent nerve (\cref{repr-ex-infty-t-commute}). 
To do so, we use the following fact about cubical natural transformations.
\begin{proposition} \label{cubical-nat-transf-to-qcat-nat-transf}
    Let $F, G \from \cubmc{C} \to \cubmc{D}$ be cubical functors.
    Any natural transformation $\alpha$ from $F$ to $G$ induces a natural transformation
    \[ \alpha \from \cohnerve \cubmc{C} \times \simp{1} \to \cohnerve \cubmc{D} \]
    from $\cohnerve F$ to $\cohnerve G$.
\end{proposition}
\begin{proof}
Given an ordinary category $\cubmc O$, regarded as a cubical category
with discrete mapping spaces, define a cubical category $\cubmc C\times\cubmc O$
as follows: Objects are the pairs $\pr{c,x}$, where $c\in\cubmc C$
and $x\in\cubmc O$. Mapping spaces are given by
\[
\pr{\cubmc C\times\cubmc O}\pr{\pr{c_{0},x_{0}},\pr{c_{1},x_{1}}}=\cubmc C\pr{c_{0},c_{1}}\otimes\cubmc O\pr{x_{0},x_{1}}.
\]
We define the composition and the identity maps of $\cubmc C\times \cubmc{O}$
so that the maps $\cubmc C\pr{c_{0},c_{1}}\otimes\cubmc O\pr{x_{0},x_{1}}\to\cubmc C\pr{c_{0},c_{1}}$
and $\cubmc C\pr{c_{0},c_{1}}\otimes\cubmc O\pr{x_{0},x_{1}}\to\cubmc O\pr{x_{0},x_{1}}$
determine cubical functors from $\cubmc C\times\cubmc O$ to $\cubmc C$
and $\cubmc O$. 
These maps exhibit $\cubmc C\times\cubmc O$
as a product of $\cubmc C$ and $\cubmc O$ in $\Cat_{\square}$. (Note
that the monoidal category $\pr{\mathsf{cSet}, \otimes}$ is \textit{not} cartesian
monoidal, so cartesian products of cubical categories are in general
difficult to describe.)

Unwinding the definitions, we can identify enriched natural transformations
$F_{0}\Rightarrow F_{1}$ of cubical functors $\cubmc C\to\cubmc D$ with
cubical functors $H:\cubmc C\times[1]\to\cubmc D$ such that $\restr{H}{\cubmc{C} \times \{ \varepsilon \} }=F_{\varepsilon}$.
The claim then follows from the isomorphism of cubical categories
$\cohnerve\pr{\cubmc C\times[1]}\cong \cohnerve\pr{\cubmc C}\times \cohnerve\pr{[1]}\cong \cohnerve\pr{\cubmc C}\times\Delta^{1}$.
\end{proof}

We highlight one subtle point: since the geometric product of cubical sets is not symmetric, the opposite of a cubical category $\cubmc C$ is not well-defined. 
Therefore, we can speak of the \emph{covariant} representable functors of $\cubmc C$, but not the \emph{contravariant} ones, much less the bivariant hom-functor.
Nonetheless, the mapping spaces of $\cubmc C$ determine an ordinary functor $\cubmc{C}(-, -) \from \cubmc C_{0}^{\op}\times \cubmc{C}_0 \to \cKan$, and we have the following corollary:
\begin{corollary} \label{repr-ex-infty-t-commute}
    Let $X$ be an object in a locally Kan cubical category $\cubmc C$.
    \begin{enumerate}
        \item The functor
        \[
        \cohnerve\cubmc C \xrightarrow{\cohnerve\pr{\cubmc C\pr{X,-}}}\cohnerve\cKan \simeq \cohnerve[\Delta] \sKan
        \]
        is naturally equivalent to the representable $(\cohnerve \cubmc{C})(X, -)$.
        \item The functor
        \[
        \cohnerve (\cubmc C_0)^\op \times \cohnerve(\cubmc C_0) \xrightarrow{\cohnerve \cubmc C(\mathord{-}, -) }\cohnerve \cKan \simeq \cohnerve[\Delta] \sKan
        \]
	is naturally equivalent to the restriction of the hom-functor $(\cohnerve \cubmc{C})(\mathord{-}, -)$ to $\cohnerve (\cubmc{C}_0)^\op \times \cohnerve (\cubmc{C}_0) = \nerve (\cubmc{C}_0)^\op\times \nerve (\cubmc{C}_0)$.
    \end{enumerate}
\end{corollary}
\begin{proof}
	Using \cref{cubical-nat-transf-to-qcat-nat-transf}, we can replace $\cubmc C$ appropriately (e.g.\ by $U_{\bullet}\Ex_{\bullet}^{\infty}T_{\bullet}\cubmc C$) and assume that $\cubmc C = U_{\bullet}\cubmc D$ for some locally Kan
simplicial category $\cubmc D$. 
In this case, we have diagrams
\[ \begin{tikzcd}[column sep = 4.2em, row sep = 2.8em]
    \cohnerve[\Delta] \cubmc{D} \ar[r, "{\cohnerve[\Delta] \cubmc{D}(X, -)}"] \ar[d, "\cong"'] & \cohnerve[\Delta] \sKan \ar[d, "\simeq"] \\
    \cohnerve \cubmc{C} \ar[r, "{\cohnerve \cubmc{C}(X, -)}"] & \cohnerve \cKan
\end{tikzcd} \qquad \begin{tikzcd}[column sep = 4.2em, row sep = 2.8em]
    \cohnerve[\Delta] (\cubmc{D}_0)^\op \times \cohnerve[\Delta] (\cubmc{D}_0) \ar[r, "{\cohnerve[\Delta] \cubmc{D}(-, \, -)}"] \ar[d, "\cong"'] & \cohnerve[\Delta] \sKan \ar[d, "\simeq"] \\
    \cohnerve (\cubmc{C}_0)^\op \times \cohnerve (\cubmc{C}_0) \ar[r, "{\cohnerve \cubmc{C}(-, \, -)}"] & \cohnerve \cKan
\end{tikzcd} \]
commuting up to equality.
The top arrows admit the desired equivalence by \cite[03N7]{kerodon}.
\end{proof}

\end{appendices}
\section*{Acknowledgment}

K.A.~acknowledges the support of JSPS KAKENHI Grant Number 24KJ1443.
D.C.~acknolwedges the support of the NSERC Canada Graduate Scholarship.
K.K.~acknowledges the support of the NSERC Discovery Grant ``Applied Homotopy Theory''.

 \bibliographystyle{amsalphaurlmod}
 \bibliography{all-refs.bib}

\end{document}